\documentclass[reqno]{amsart}
\usepackage{amssymb}
\usepackage{graphicx}
\usepackage{amsmath}
\usepackage{mathbbol}
\usepackage{esint}
\usepackage[margin=1.25in]{geometry}

\usepackage[usenames, dvipsnames]{color}
\usepackage{verbatim}
\usepackage{mathrsfs}
\usepackage{bm}
\usepackage{cite}
\usepackage{color}
\allowdisplaybreaks[4]

\numberwithin{equation}{section}

\newtheorem{theorem}{Theorem}[section]
\newtheorem{corollary}[theorem]{Corollary}
\newtheorem{lemma}[theorem]{Lemma}
\newtheorem{prop}[theorem]{Proposition}

\theoremstyle{definition}
\newtheorem{remark}[theorem]{Remark}

\theoremstyle{definition}

\theoremstyle{definition}

\makeatletter
\def\dashint{\operatorname%
	{\,\,\text{\bf-}\kern-.98em\DOTSI\intop\ilimits@}}
\makeatother

\def\\det{\text{\det}}

\def\.5{\frac{1}{2}}

\newcommand{\RN}[1]{%
	\textup{\uppercase\expandafter{\romannumeral#1}}%
}

\renewcommand{\epsilon}{\varepsilon}

\newcounter{marnote}

\begin{document}
	
		\title[Stress concentration problem in Navier-Stokes flow]{Stress concentration between two adjacent rigid particles in Navier-Stokes flow}
	
	\author[H.G. Li]{Haigang Li}
	\address[H.G. Li]{School of Mathematical Sciences, Beijing Normal University, Laboratory of Mathematics and Complex Systems, Ministry of Education, Beijing 100875, China.}
	\email{hgli@bnu.edu.cn}
	
   \author[P.H. Zhang]{PeiHao Zhang}
	\address[P.H. Zhang]{School of Mathematical Sciences, Beijing Normal University, Laboratory of Mathematics and Complex Systems, Ministry of Education, Beijing 100875, China.}
	\email{phzhang@mail.bnu.edu.cn}


	\date{\today} 
	
\begin{abstract}
In this paper we investigate the stress concentration problem that occurs when two convex rigid particles are closely immersed in a fluid flow. The governing equations for the fluid flow are the stationary incompressible Navier-Stokes equations. We establish precise upper bounds for the gradients and second-order derivatives of the fluid velocity as the distance between particles approaches zero, in dimensions two and three. The optimality of these blow-up rates of the gradients is demonstrated by deriving corresponding lower bounds. New difficulties arising from the nonlinear term in the Navier-Stokes equations is overcome. Consequently, the blow up rates of the Cauchy stress are studied as well.  
\end{abstract}

	\maketitle

\section{Introduction}

\subsection{Formulation of Fluid–Rigid Particle System}

The main objective of this paper is to investigate the phenomenon of stress concentration between two closely adjacent rigid particles immersed in a fluid flow, with the governing equations for the fluid being the stationary incompressible Navier-Stokes system. 

Let $D\subset\mathbb{R}^{d}$ be an open bounded set, where $d=2,3$. We denote by $D_{1}\subset{D}$ and $D_{2}\subset{D}$ the domains occupied by two rigid particles and by $\Omega:=D\setminus\overline{D_1\cup D_2}$ the domain occupied by the fluid. In order to investigate the interaction between these two adjacent rigid particles, we assume that the distance $\epsilon:=\mbox{dist}(D_{1},D_{2})$ is sufficiently small, and $\overline{D}_{1},\overline{D}_{2}\subset D$ are located far from the boundary $\partial D$, specifically, $\mbox{dist}(D_{1}\cup D_{2},\partial D)>\kappa_{0}$ where $\kappa_{0}$ is a positive constant independent of $\varepsilon$. 

The full system of equations that model the suspension problem with two adjacent rigid particles in a stationary Navier-Stokes flow can be written as follows:
\begin{equation}\label{sto}
	\left\{\begin{split}
\mu\, \Delta{\bf u}&=\nabla p+{\bf u}\cdot\nabla{\bf u},\quad\nabla\cdot {\bf u}=0,&\hbox{in}&~~\Omega,\\
{\bf u}|_{+}&={\bf u}|_{-}&\hbox{on}&~~\partial{D_{i}},~i=1,2,\\
{\bf e}({\bf u})&=0 &\hbox{in}&~~D_{i},~i=1,2,\\
{\bf u}&={\boldsymbol\varphi} &\hbox{on}&~~\partial{D},
\end{split}\right.
\end{equation}
where the constant $\mu>0$ represents the viscosity of the fluid, ${\bf u}$ is the fluid velocity, $p$ is the pressure, the strain tensor ${\bf e}({\bf u})$ is defined as the symmetric part of $\nabla{\bf u}$, and the subscript $\pm$ indicates the limit from outside and inside the domain, respectively. Furthermore, we assume that each rigid particle is in equilibrium, implying that the total force and torque on each boundary are zero (see \cite{AKKY}), that is, 
\begin{align}\label{sto-2}
 \int_{\partial{D}_{i}} {\boldsymbol\psi}_j\cdot \sigma[{\bf u}, p]\nu=0,~i=1,2,~j=1,\dots\frac{d(d+1)}{2}.
\end{align}
Here $\sigma[{\bf u},p]=2\mu\, {\bf e}({\bf u})-p\,\mathbb{I}$ represents the Cauchy stress tensor and $\mathbb{I}$ is the identity matrix, $\nu$ is the  unit outer normal vector of $D_{i}$, and $\{{\boldsymbol\psi}_{j}\}_{j=1}^{d(d+1)/2}=\{\boldsymbol{e}_{i},~x_{k}\boldsymbol{e}_{l}-x_{l}\boldsymbol{e}_{k}~|~1\leq\,i\leq\,d,~1\leq\,l<k\leq\,d\}$ is a basis of
$$\Psi:=\Big\{{\boldsymbol\psi}\in C^{1}(\mathbb{R}^{d};\mathbb{R}^{d})~|~e({\boldsymbol\psi}):=\frac{1}{2}(\nabla{\boldsymbol\psi}+(\nabla{\boldsymbol\psi})^{\mathrm{T}})=0\Big\},$$
the linear space of rigid displacements in $\mathbb{R}^{d}$, where $\boldsymbol{e}_{1},\dots,\boldsymbol{e}_{d}$ is the standard basis of $\mathbb{R}^{d}$.  Since $D$ is bounded, it follows from the divergence-free condition $\nabla\cdot{\bf u}=0$ and Gauss's theorem that the prescribed boundary velocity field ${\boldsymbol\varphi}$ satisfies the compatibility condition:
\begin{equation}\label{compatibility}
\int_{\partial{D}}{\boldsymbol\varphi}\cdot \nu\,=0.
\end{equation}
The existence and uniqueness of a weak solution to the stationary Navier-Stokes equation can be deduced from the work of Ladyzhenskaya \cite{Lady}, by employing an integral variational formulation and the Riesz representation theorem.  For a comprehensive exploration of the steady Navier-Stokes equations, the interested readers can refer to the monograph of Galdi \cite{GaldiBook}.

\subsection{Motivation and Background}\label{subsec1.1} 

The problem of describing the interactions among densely packed particles is one of the most challenging problems in the analysis of fluid-solid structure. This problem has been extensively studied in fluid mechanics and mathematical literature. Many results have been achieved to analyze the effective viscosity of the mixture and to understand the close-to-contact motion of particles. For instance, refer to \cite{DG} for the effective viscosity in the dilute regime in the framework of homogenization theory, and to \cite{HS} for a model of  fluid behavior with an incompressible Stokes system, where the author showed that solutions are global and no collision occurs between the spheres in finite time. There are also many  results that do not allow contact between particles by assuming that the distance between particles is not relatively small, see \cite{AGKL,BMT,GH}.

To precisely characterize this singularity of the solution, the main objective of this paper is to establish pointwise derivative estimates of the solutions in the narrow region between two neighboring particles, as the distance $\varepsilon$ tends to zero, rather than just the $L^{2}$ estimates as previously done. A recent significant advancement on this topic was made by Ammari, Kang, Kim, and Yu \cite{AKKY}. In a two-dimensional steady Stokes system, they initially derived an asymptotic representation formula for the stress and successfully captured the singularity of the stress, by employing the method of the bipolar coordinates for two adjacent disks $D_{1}$ and $D_{2}$ and showed the blow up rate of $|\nabla {\bf u}|$ is of order $\varepsilon^{-1/2}$ in dimension two. While, the study of general shape inclusions and higher dimensions cases is quite interesting and challenging, as Kang mentioned in his ICM talk \cite{K}. Fortunately, the first author and Xu \cite{LX,LX2} resolved this problem in previous work \cite{BLL,BLL2} and proved that the optimal blow up rates of the gradient of the velocity field are of order $\varepsilon^{-1/2}$ for $d=2$ and $(\varepsilon|\log\varepsilon|)^{-1}$ for $d=3$, for general convex inclusions $D_{1}$ and $D_{2}$. 

It is worth noting that these optimal blow up rates are consistent with those found in the perfect conductivity problem \cite{BLY,AKL,KangLY,AKL3} and the linear elasticity problem with hard inclusions \cite{KY,BLL,BLL2,Li2021}. The background equations in these cases are all linear, including the Laplace equation, the Lam\'e system and the Stokes equations. In problem \eqref{sto}, the presence of the nonlinear term ${\bf u}\cdot\nabla{\bf u}$ in the Navier-Stokes equations undoubtedly introduces new challenges and difficulties in analysis compared to the Stokes equations. 

The novelty of this paper is our attempt to deal with the nonlinear Navier-Stokes equations by apply the iteration method, initially  developed for studying linear elasticity problems \cite{BLL,BLL2} and Stokes equations \cite{LX,LX2}.  Undoubtedly, this requires overcoming several new difficulties and involves more intricate analysis. To the best of our knowledge, this is the first result regarding stress estimates between two adjacent rigid bodies in the stationary Navier-Stokes flow. We establish pointwise upper bounds for the gradients and second order derivatives of the fluid velocity ${\bf u}$, as well as corresponding  estimates for the pressure $p$ and its gradient $\nabla p$. Furthermore, we derive a lower bound for the gradient at the narrowest points between two particles, demonstrating the optimality of the blow up rates of the gradient of ${\bf u}$. This method is applicable to inclusions of arbitrary shape and in all dimensions. 

Due to the similarity in mechanics, this work is partially inspired by the study of physical field enhancement between two adjacent inclusions in composites, particularly in the context of electrostatic and elasto-static. This is a booming field with numerous interesting works, for example, on the conductivity problem \cite{AKL,KLiY,KY,KangLY,BT,BV,DLY1,LiYang,Ben}, and on the linear elasticity problem \cite{BJL,LLBY}. While, the suspension problem in complex fluids is one of the oldest problems in fluid mechanics, with pioneering contribution from Stokes \cite{Stokes}, Kirchhoff \cite{Kirchhoff}, and Jeffery \cite{Jeffery}. It is closely connected to Reynolds's  hydrodynamic lubrication theory  \cite{Reynolds1886}.  In the past century, Grubin proposed a coupling between Reynolds's fluid lubrication theory and Hertz's elastic contact theory, which has proven to be a powerful tool for solving intricate contact surface lubrication problems, see \cite{Oscar1987,Cheng1967}. 

This present paper would be the first step to address other problems of physical interest, such as the dynamics of multiple solid bodies in a fluid flow. In a forthcoming work, we shall address the evolution of two adjacent rigid bodies in a viscous incompressible fluid and establish a transformation mechanism between stress concentration and the rotation of particles. For the case of a single rigid solid, G\'erard-Varet and Hillairet \cite{Hill5} studied the motion of a rigid solid inside an incompressible Navier-Stokes flow in a bounded domain and proved the existence of weak solutions of Leray type, up to collision, in three dimensions. Muha, Ne\v{c}asov\'{a} and Rado\v{s}evi\'c study the regularity of weak solution to the fluid-rigid body problem with a single particle, assuming that the rigid body does not contact with the boundary and that the acceleration of the rigid body is bounded. For more interesting works on this topic, one can refer to \cite{GPG,EMT,SanMartin,Gunzburger,Hill2,DM1,DK,CN,Weinberger} and the references therein. 
  
\subsection{Assumptions on Our Domain}
Before stating our main results precisely, we will first fix the domain and notation. After a rotation of coordinates if necessary, let us denote by $D_{1}^{0}$ and $D_{2}^{0}$ a pair of relatively convex subdomains of $D$ that touch at the origin, and satisfy
   \begin{equation*}
   	D_{1}^{0}\subset\{(x', x_{d})\in\mathbb R^{d}|~ x_{d}>0\},\quad D_{2}^{0}\subset\{(x', x_{d})\in\mathbb R^{d}|~ x_{d}<0\},
   \end{equation*}
with $\{x_d=0\}$ being their common tangent plane. Here and throughout this paper, we use prime superscripts to represent $(n-1)$-dimensional variables and domains, such as $x'$ and $B'$. We translate $D_{i}^{0}$ ($i=1,\, 2$) by $\pm\frac{\varepsilon}{2}$ along the $x_{d}$-axis in the following manner:
   \begin{equation*}
   	D_{1}^{\varepsilon}:=D_{1}^{0}+(0',\frac{\varepsilon}{2})\quad \text{and}\quad D_{2}^{\varepsilon}:=D_{2}^{0}+(0',-\frac{\varepsilon}{2}).
   \end{equation*}
For simplicity of notation, we denote $D_{i}:=D_{i}^{\varepsilon}$, $i=1,2$, by dropping the superscript $\varepsilon$. Because we shall establish estimates for second-order derivatives of ${\bf u}$, we assume that our domains are of class $C^3$ and that the $C^{3}$ norms of $\partial{D}_1,\partial{D}_2$ and $\partial{D}$ are bounded by a positive constant $\kappa_{1}$, which is independent of $\varepsilon$. 

There exists a constant $R$, independent of $\varepsilon$, such that the portions of $\partial D_1$ and $\partial D_2$ near the origin can be expressed, respectively, by
 \begin{equation*}
 	x_d=\frac{\varepsilon}{2}+h_1(x')\quad\text{and}\quad x_d=-\frac{\varepsilon}{2}-h_2(x'),\quad \text{for}~ |x'|\leq 2R,
 \end{equation*}
 where $h_1,h_2\in C^{3}(B'_{2R}(0'))$ satisfy 
 \begin{align}
 	&-\frac{\varepsilon}{2}-h_{2}(x') <\frac{\varepsilon}{2}+h_{1}(x'),\quad\mbox{for}~~ |x'|\leq 2R,\label{h1-h2}\\
 	&h_{1}(0')=h_2(0')=0,\quad \partial_{x'} h_{1}(0')=\partial_{x'}h_2(0')=0,\label{h1h1}\\
 	&h_{1}(x')=h_2(x')=\frac{\kappa}{2}|x'|^{2}+O(|x'|^{3}),\quad\mbox{for}~~|x'|<2R,\label{h1h14}
 \end{align}
 with a constant $\kappa>0$. 
Throughout this paper, we say a constant $C$ is {\em{universal}}  if it depends only on $d,\mu,\kappa_{0},\kappa_{1}$, and $\kappa$, but independent of $\varepsilon$. For the sake of clarity, we focus only on dimensions $d=2,3$, these two important and physically relevant dimensions. We first sate our main results in dimension three in the following. 
 
 \subsection{Upper Bound Estimates in 3D} \label{sub1.2}
 For $0\leq r\leq 2R$, let us define the neck region between $\partial{D}_{1}$ and $\partial{D}_{2}$ by
 \begin{equation*}
 	\Omega_r:=\left\{(x',x_{d})\in \Omega~:~ -\frac{\varepsilon}{2}-h_2(x')<x_{d}<\frac{\varepsilon}{2}+h_1(x'),~|x'|<r\right\}.
 \end{equation*}
We adapt the iteration approach developed in \cite{BLL,BLL2,LX2} to nonlinear Navier-Stokes equations, overcoming the difficulties caused by the nonlinear terms ${\bf u}\cdot\nabla{\bf u}$, and derive the following piontwise upper bound estimates.

\begin{theorem}\label{mainthm}(Upper Bounds in 3D)
	Assume that $D_1,D_2,D$ and $\varepsilon$ are defined as above, with a given ${\boldsymbol\varphi}\in C^{2,\alpha}(\partial D;\mathbb R^3)$, $0<\alpha<1$. Let ${\bf u}\in W^{1,2}(D;\mathbb R^3)\cap C^2(\bar{\Omega};\mathbb R^3)$ and $p\in L^2(D)\cap C^1(\bar{\Omega})$ be the solution to \eqref{sto}--\eqref{compatibility}. Then for sufficiently small $0<\varepsilon<1/2$, we have the following assertions:
	
(i) For the gradient of ${\bf u}$ and $p$, 
\begin{equation*}
|\nabla{\bf u}(x)|\leq
		\frac{C (1+|\log\varepsilon||x'|)}{|\log\varepsilon|(\varepsilon+|x'|^{2})}\|{\boldsymbol\varphi}\|_{C^{2,\alpha}(\partial D)}, \quad~x\in\Omega_{R},
\end{equation*}
and
\begin{equation}\label{xt2p}
\inf_{c\in\mathbb{R}}\|p+c\|_{C^0(\bar{\Omega}_{R})}\leq \frac{C}{\varepsilon|\log\varepsilon|}\|{\boldsymbol\varphi}\|_{C^{2,\alpha}(\partial D)},
\end{equation}
and $\|\nabla{\bf u}\|_{L^{\infty}(\Omega\setminus\Omega_{R})}+\inf_{c\in\mathbb{R}}\|p+c\|_{L^{\infty}(\Omega\setminus\Omega_{R})}\leq\,C\|{\boldsymbol\varphi}\|_{C^{2,\alpha}(\partial D)}$, where $C$ is a universal constant; 
	
	(ii) For the second order derivatives of ${\bf u}$ and the gradient of $p$,
	\begin{align*}
		|\nabla^2{\bf u}(x)|+|\nabla p(x)|\leq
		\frac{C(1+|\log\varepsilon||x'|)}{|\log\varepsilon|(\varepsilon+|x'|^2)^{2}}\|{\boldsymbol\varphi}\|_{C^{2,\alpha}(\partial D)},\quad~x\in\Omega_{R},
	\end{align*}
	and $\|\nabla^2{\bf u}\|_{L^{\infty}(\Omega\setminus\Omega_{R})}+\|\nabla p\|_{L^{\infty}(\Omega\setminus\Omega_{R})}\leq C\|{\boldsymbol\varphi}\|_{C^{2,\alpha}(\partial D)},$
where $C$ is a universal constant. 
\end{theorem}

\begin{remark}\label{mainthmstokes}
We give some remarks on Theorem \ref{mainthm}. 

$(i)$ The estimate of $p$ in \eqref{xt2p} is optimal, because its order now matches that of $|\nabla{\bf u}|$. It is worth pointing out that \eqref{xt2p} is also holds for the Stokes equations, thereby improving the previously obtained order in \cite{LX2} from $(\epsilon^{3/2}|\log\epsilon|)^{-1}$ to $(\epsilon|\log\epsilon|)^{-1}$.

$(ii)$ In particular, we have
	\begin{equation*}
		\|\nabla{\bf u}\|_{L^{\infty}(\Omega)}\leq \frac{C}{\varepsilon|\log\varepsilon|}\|{\boldsymbol\varphi}\|_{C^{2,\alpha}(\partial D)}.
	\end{equation*}
The optimality of this blow up rate of $|\nabla{\bf u}|$ will be proved by Theorem \ref{theolow}  below. 
\end{remark}

As an immediate consequence of Theorem \ref{mainthm}, we obtain the blow up estimate of the Cauchy stress tensor 
$$\sigma[{\bf u},p]=2\mu\, {\bf e}({\bf u})-p\,\mathbb{I}.$$

\begin{corollary}\label{mainthmsigma}(Estimates of the Cauchy Stress)
Under the assumptions in Theorem \ref{mainthm}, we have  
	\begin{equation*}
		\inf_{c\in\mathbb{R}}|\sigma[{\bf u},p+c]|\leq 
	\frac{C}{\varepsilon|\log\varepsilon|}\|{\boldsymbol\varphi}\|_{C^{2,\alpha}(\partial D)},\quad~\text{in}~~\Omega_R.
\end{equation*}
\end{corollary}

\subsection{The Optimality of the Blow Up Rates}

The lower bounds of $|\nabla {\bf u}(x)|$ presented below demonstrate the optimality of the blow up rates of $|\nabla {\bf u}(x)|$ as established in Theorem \ref{mainthm}. To this end, we require additional symmetric assumptions on the domain and the prescribed boundary data. 

\begin{theorem}(Lower Bounds)\label{theolow}
	Assume that $D=B_4(0)$, and $D_1=B_{1}(0',1+\frac{\varepsilon}{2})$ and $D_2=B_{1}(0',-1-\frac{\varepsilon}{2})$ are two unit balls. Let ${\bf u}\in W^{1,2}(D;\mathbb R^{3})\cap C^2(\bar{\Omega};\mathbb R^{3})$ and $p\in L^2(D)\cap C^1(\bar{\Omega})$ be the solution to \eqref{sto}--\eqref{compatibility} with $\boldsymbol{\varphi}=(0',x_{3})$. Then for sufficiently small $0<\varepsilon<1/2$, there holds
	\begin{equation*}
		|\nabla {\bf u}(0',x_3)|\ge \frac{1}{C\varepsilon|\log\varepsilon|},\quad\quad\quad~~\text{for}~|x_{3}|\leq \varepsilon.
	\end{equation*}
\end{theorem}

The main results for two-dimensional case are presented in Section \ref{sec5}. For dimensions $d\ge 4$, with minor modifications and in conjunction with the auxiliary functions constructed in \cite{LX2}, we can obtain the following estimates as well
$$||\nabla{\bf u}||_{L^\infty(\Omega)}\leq \frac{C}{\varepsilon^{3/2}}\|{\boldsymbol\varphi}\|_{C^{2,\alpha}(\partial D)},\quad\,d\geq4.$$ 

The rest of this paper is organized as follows. In Section \ref{sec2}, we outline our strategy by decomposing the solution of \eqref{sto} into $d(d+1)+1$ parts, $({\bf u}_{i}^{\alpha},p_{i}^{\alpha})$ with $i=1,2$, $\alpha=1,2,\dots,\frac{d(d+1)}{2}$, and incorporating the nonlinear term ${\bf u}\cdot\nabla{\bf u}$ in the Navier-Stokes equations that $({\bf u}_{1}^{d},p_{1}^{d})$ satisfies, while other components satisfy the Stokes equations with different boundary data. We will establish the gradient estimates and the second order derivatives estimates for each component $({\bf u}_{i}^{\alpha},p_{i}^{\alpha})$ and discuss  the main challenges that need to be addressed when adopting the previous iteration method in \cite{LX} to this nonlinear Navier-Stokes equations. Subsequently, we list the main ingredients and use them to prove Theorem \ref{mainthm} at the end of Section \ref{sec2}. In Section \ref{sec3}, we prove some fundamental framework results for the Navier-Stokes equations in order to apply the iteration argument.  We derive the $W^{1,\infty}$ and $W^{2,\infty}$ estimates in the narrow region by employing a bootstrap argument and  rescaling technique. In Section \ref{sec4}, we apply the results in Section \ref{sec3} to prove the estimates for $({\bf u}_{1}^{3},p_{1}^{3})$ and establish the estimates for $|C_{1}^{\alpha}-C_{2}^{\alpha}|$. Section \ref{sec6} is dedicated to proving Theorems \ref{theolow} concerning the lower bounds of $|\nabla {\bf u}|$ in dimension three by comparing with the lower bound results for the Stokes system. In Section \ref{sec5}, we present the key  estimates for 2D and prove the Theorem \ref{mainthm2D} for the upper bounds and Theorem \ref{theolow2d} for the lower bounds.

\section{Proof Strategy of Theorem \ref{mainthm} and New Ingredients in 3D}\label{sec2}

In this section we adopt the strategy outlined in \cite{LX,LX2} to solve the suspension problem of the Navier-Stokes equations \ref{sto}. We divide the solution into $d(d+1)+1$ parts, with the nonlinear term ${\bf u}\cdot\nabla{\bf u}$ only included in the equation of $({\bf u}_{1}^{d},p_{1}^{d})$, for $d=2,3$. However, adapting our approach, initially developed for linear elliptic equations and systems, to deal with this nonlinear equation of $({\bf u}_{1}^{d},p_{1}^{d})$ presents significant technical challenges. This results in numerous new and critical obstacles, which constitute the primary contribution of this paper. We demonstrate these difficulties using the three-dimensional case as an example in the following subsections. There is another interesting issue in two-dimensional case discussed in Section \ref{sec5}.

\subsection{Decomposition of the Solution}
It follows from the fourth line in \eqref{sto}, ${\bf e}({\bf u})=0$, that ${\bf u}$ is a linear combination of ${\boldsymbol\psi}_{\alpha}$ in each inclusion $D_i$, $i=1,2$, where $\{{\boldsymbol\psi}_{\alpha}\}_{\alpha=1}^{6}$ is a basis of $\Psi$:
{\small $${\boldsymbol\psi}_{1}=\begin{pmatrix}
	1 \\
	0\\
	0
\end{pmatrix},
{\boldsymbol\psi}_{2}=\begin{pmatrix}
	0\\
	1\\
	0
\end{pmatrix},
{\boldsymbol\psi}_{3}=\begin{pmatrix}
	0\\
	0\\
	1
\end{pmatrix},
{\boldsymbol\psi}_{4}=\begin{pmatrix}
	x_{2}\\
	-x_{1}\\
	0
\end{pmatrix},
{\boldsymbol\psi}_{5}=\begin{pmatrix}
	x_{3}\\
	0\\
	-x_{1}
\end{pmatrix},
{\boldsymbol\psi}_{6}=\begin{pmatrix}
	0\\
	x_{3}\\
	-x_{2}
\end{pmatrix}.
$$}
Namely,
\begin{equation}\label{introC}
	{\bf u}=\sum_{\alpha=1}^{6}C_{i}^{\alpha}{\boldsymbol\psi}_{\alpha},\quad\mbox{in}~D_i,\quad i=1,2.
\end{equation}
Note that these twelve constants $C_{i}^{\alpha}$ in \eqref{introC} are free, which will be  determined later by the solution $({\bf u},p)$. Due to the continuity of the transmission condition on $\partial{D}_{i}$, the solution to equation \eqref{sto} in $\Omega$ can be decomposed into thirteen parts in the following way
\begin{align}\label{udecom}
	{\bf u}(x)&=\sum_{i=1}^{2}\sum_{\alpha=1}^{6}C_i^{\alpha}{\bf u}_{i}^{\alpha}(x)+{\bf u}_{0}(x),\quad~ \mbox{and}~
	p(x)=\sum_{i=1}^{2}\sum_{\alpha=1}^{6}C_i^{\alpha}p_{i}^{\alpha}(x)+p_{0}(x),
\end{align}
where, especially, ${\bf u}_{1}^{3}\in{C}^{2}(\overline{\Omega};\mathbb R^3),p_{1}^{3}\in{C}^{1}(\overline{\Omega})$ satisfy the Navier-Stokes equations
\begin{equation}\label{equ_v13}
	\begin{cases}
		\nabla\cdot\sigma[{\bf u}_{1}^3,p_{1}^{3}]={\bf u}\cdot\nabla{\bf u},\quad\nabla\cdot {\bf u}_{1}^{3}=0,&\mathrm{in}~\Omega,\\
		{\bf u}_{1}^{3}={\boldsymbol\psi}_{3},&\mathrm{on}~\partial{D}_{1},\\
		{\bf u}_{1}^{3}=0,&\mathrm{on}~\partial{D_{2}}\cup\partial{D}, 
	\end{cases}
\end{equation}
 while, for $(i,\alpha)\neq(1,3)$, ${\bf u}_{i}^{\alpha}\in{C}^{2}(\overline{\Omega};\mathbb R^3),p_{i}^{\alpha}\in{C}^{1}(\overline{\Omega})$, respectively, satisfy the Stokes equations 
\begin{equation}\label{equ_v1}
	\begin{cases}
		\nabla\cdot\sigma[{\bf u}_{i}^\alpha,p_{i}^{\alpha}]=0,~\nabla\cdot {\bf u}_{i}^{\alpha}=0,&\mathrm{in}~\Omega,\\
		{\bf u}_{i}^{\alpha}={\boldsymbol\psi}_{\alpha},\quad\,&\mathrm{on}~\partial{D}_{i},\\
		{\bf u}_{i}^{\alpha}=0,\quad\quad&\mathrm{on}~\partial{D_{j}}\cup\partial{D},j\neq i,
	\end{cases}~~i=1,2,
\end{equation}
and ${\bf u}_{0}\in{C}^{2}(\overline{\Omega};\mathbb R^3),p_0\in{C}^{1}(\overline{\Omega})$ satisfy 
\begin{equation}\label{equ_u0}
	\begin{cases}
		\nabla\cdot\sigma[{\bf u}_{0},p_0]=0,\quad\nabla\cdot {\bf u}_{0}=0,&\mathrm{in}~\Omega,\\
		{\bf u}_{0}=0,&\mathrm{on}~\partial{D}_{1}\cup\partial{D_{2}},\\
		{\bf u}_{0}={\boldsymbol\varphi},&\mathrm{on}~\partial{D}.
	\end{cases}
\end{equation}

Here we only consider the special case that $C_{1}^{3}=1$ to illustrate our method. For the other cases, a different decomposition may be required. To solve nonlinear problem \eqref{equ_v13}, we recombine \eqref{udecom} as
\begin{align*}
	{\bf u}(x)=C_1^{3}{\bf u}_{1}^{3}(x)+{\bf u}^{\#3}(x),\quad~ \mbox{and}~
	p(x)=C_1^{3}p_{1}^{3}(x)+p^{\#3}(x),
\end{align*}
where
\begin{align}\label{udecom3}
	{\bf u}^{\#3}(x):=\sum_{(i,\alpha)\neq (1,3)}C_i^{\alpha}{\bf u}_{i}^{\alpha}(x)+{\bf u}_{0}(x),\quad~ \mbox{and}~
	p^{\#3}(x):=\sum_{(i,\alpha)\neq (1,3)}C_i^{\alpha}p_{i}^{\alpha}(x)+p_{0}(x),
\end{align}
represent the sum of all components which satisfy the linear Stokes equations. Thus, the Navier-Stokes equations \eqref{equ_v13} that $({\bf u}_{1}^{3},p_{1}^{3})$ satisfies can be rewritten as 
\begin{equation}\label{equ_u13}
	\left\{\begin{split}
\mu \Delta{\bf u}_{1}^3=&\,\nabla p_{1}^{3}+(C_1^{3})^{2}{\bf u}_{1}^{3}\cdot\nabla\,{\bf u}_{1}^{3}+C_1^{3}{\bf u}_{1}^{3}\cdot\nabla{\bf u}^{\#3}\\
		&\quad\quad+C_1^{3}{\bf u}^{\#3}\cdot\nabla\,{\bf u}_{1}^{3}+{\bf u}^{\#3}\cdot\nabla{\bf u}^{\#3},&\mathrm{in}&~\Omega,\\
		\nabla\cdot {\bf u}_{1}^{3}=&\,0,&\mathrm{in}&~\Omega,\\
		{\bf u}_{1}^{3}=&\,{\boldsymbol\psi}_{3},~\mathrm{on}~\partial{D}_{1},
		\quad{\bf u}_{1}^{3}=0,~\mathrm{on}~\partial{D_{2}}\cup\partial{D}.
	\end{split}\right.
\end{equation}
Therefore, based on the definitions of ${\bf u}^{\#3}$ and $p^{\#3}$ in \eqref{udecom3}, in order to estimate $({\bf u}_{1}^{3},p_{1}^{3})$,  we need to first establish precise estimates for $C_{i}^{\alpha}$ and $({\bf u}_{i}^{\alpha},p_{i}^{\alpha})_{(i,\alpha)\neq (1,3)}$ as well as $({\bf u}_{0},p_{0})$.

Firstly, by the trace theorem along with the fact that
$\|{\bf u}\|_{H^{1}(\Omega\setminus\Omega_{R})}\leq\,C$, similarly as in \cite{BLL}, we have
\begin{prop}\label{lemma11}
	Let $C_i^\alpha$ be defined in \eqref{udecom}. Then
	$$|C_i^\alpha|\leq C,~i=1,2;~~\alpha=1,2,\dots,6.$$
\end{prop}
Secondly, it is important to note that the components $({\bf u}_{i}^{\alpha},p_{i}^{\alpha})_{(i,\alpha)\neq (1,3)}$ and $({\bf u}_{0},p_{0})$ all satisfy the Stokes equations, and their gradient estimates have been established in the previous work of the first author and Xu \cite{LX2}. Therefore, we can consider the terms involving ${\bf u}^{\#3}$ appearing in \eqref{equ_u13} as known terms. However, despite this, there are still several difficulties in studying the singular behavior of the gradient of the solution to the nonlinear Navier-Stokes equations of type \eqref{equ_u13}. Due to its independent interest, we focus on discussing this problem in the next two sections, Section \ref{sec3} and \ref{sec4}.

To estimate ${\bf u}^{\#3}$ and $p^{\#3}$, for the reader's convenience, we present the estimates for the Stokes equations \eqref{equ_v1} and \eqref{equ_u0} as derived in \cite{LX2}. These include the explicit auxiliary functions $({\bf v}_{i}^{\alpha},\overline{p}_i^{\alpha})$ in $\Omega_{2R}$, as well as some optimal improvements on the estimates of the pressure $p_{i}^{\alpha}$. Since the regularity theory in $\Omega\setminus\Omega_{R}$ is standard, our focus is on establishing the blow up estimates of $(\nabla{\bf u}_{i}^{\alpha},p_{i}^{\alpha})$ and $(\nabla^{2}{\bf u}_{i}^{\alpha},\nabla p_{i}^{\alpha})$ in the narrow region $\Omega_{2R}$.

\subsection{Improved Estimates for $p_{i}^{\alpha}$ in the Linear Parts}
In what follows, we use $\delta(x')$ to represent the vertical distance between $\partial{D}_{1}$ and $\partial{D}_{2}$ at position $x'$, namely,
\begin{align*}
	\delta(x'):=\epsilon+h_{1}(x')+h_{2}(x'),\quad\mbox{for}\, |x'|\leq 2R.
\end{align*}
For $x\in\Omega_{R}$ and $r\leq\,R$, denote the narrow region by
$$\Omega_{r}(x):=\left\{(y',y_{d})\big| -\frac{\varepsilon}{2}-h_{2}(y')<y_{d}
<\frac{\varepsilon}{2}+h_{1}(y'),\,|y'-x'|<r \right\}.$$
To clearly illustrate our idea and avoid unnecessary computational difficulties, we assume, for simplicity, that $h_{1}$ and $h_{2}$ are quadratic and symmetric with respect to the plane $\{x_{3}=0\}$, say, $h_1(x')=h_2(x')=\frac{1}{2}|x'|^2$ for $|x'|\leq 2R$.  

We introduce the Keller-type function $k\in\,C^{3}(\mathbb{R}^{3})$ satisfying
\begin{equation*}
	k(x)=\frac{x_{3}}{\delta(x')},\quad\hbox{in}\ \Omega_{2R},
\end{equation*}
and  $k(x)=\frac{1}{2}$ on $\partial D_{1}$,  $k(x)=-\frac{1}{2}$ on $\partial D_{2}$, $k(x)=0$ on $\partial D$ with $\|k(x)\|_{C^{3}(\mathbb{R}^{3}\backslash\Omega_{R})}\leq C$. As in \cite{LX2}, we use this Keller-type function to construct  a family of auxiliary functions ${\bf v}_{1}^{\alpha}\in C^{3}(\Omega;\mathbb R^3)$, such that ${\bf v}_{1}^{\alpha}={\bf u}_{1}^{\alpha}={\boldsymbol\psi}_{\alpha}$ on $\partial{D}_{1}$ and ${\bf v}_{1}^{\alpha}={\bf u}_{1}^{\alpha}=0$ on $\partial{D}_{2}\cup\partial{D}$. Specifically, we construct  
\begin{align}\label{v1alpha}
	{\bf v}_{1}^{\alpha}=\boldsymbol\psi_{\alpha}\Big(k(x)+\frac{1}{2}\Big)+\boldsymbol\psi_{3}x_{\alpha}\Big(k(x)^2-\frac{1}{4}\Big),\quad\hbox{in}\ \Omega_{2R}, \quad\alpha=1,2,
\end{align}
which satisfies $\nabla\cdot{\bf v}_1^\alpha=0$ in $\Omega_{2R}$, and the associated pressure function $\overline{p}_1^{\alpha}\in C^{1}(\Omega)$ with
\begin{equation}\label{defp113D}
	\overline{p}_1^{\alpha}=\frac{2\mu x_{\alpha}}{\delta(x')}k(x),\quad\mbox{in}~\Omega_{2R},
\end{equation}
to make $|\mu\Delta{\bf v}_{1}^{\alpha}-\nabla\overline{p}_1^{\alpha}|$ be as small as possible. 
For ${\bf v}_{2}^{\alpha}$, we only need to replace $k(x)=\frac{x_3}{\delta(x')}$ by $\widetilde{k}(x)=-\frac{x_3}{\delta(x')}$ in $\Omega_{2R}$. Therefore, in the sequel we only discuss the case for $i=1$, and omit the case for $i=2$. The following results are from \cite{LX2}, except an improved estimate of $p_{i}^{\alpha}$.

\begin{prop}( An Improvement of \cite[Proposition 2.1]{LX2}\label{propu113D1})
	Let ${\bf u}_{i}^{\alpha}\in{C}^{2}(\Omega;\mathbb R^3),~p_{i}^{\alpha}\in{C}^{1}(\Omega)$ be the solution to \eqref{equ_v1} for $i,\alpha=1,2$. Then there holds
	\begin{equation*}
		\|\nabla({\bf u}_{i}^{\alpha}-{\bf v}_{i}^{\alpha})\|_{L^{\infty}(\Omega_{\delta(x')/2}(x))}\leq C,\quad \,x\in\Omega_{R},
	\end{equation*}
	and 
	\begin{equation*}
		\|\nabla^2({\bf u}_{i}^{\alpha}-{\bf v}_{i}^{\alpha})\|_{L^{\infty}(\Omega_{\delta(x')/2}(x))}+\|\nabla (p_{1}^{\alpha}-\overline{p}_{1}^{\alpha})\|_{L^{\infty}(\Omega_{\delta(x')/2}(x))}\leq\frac{C}{\delta(x')},\quad \,x\in\Omega_{R}.
\end{equation*}

Consequently, in view of \eqref{v1alpha} and \eqref{defp113D},
	\begin{align*}
		\frac{1}{C\delta(x')}\leq|\nabla {\bf u}_{i}^{\alpha}(x)|\leq \frac{C}{\delta(x')},\quad|\nabla^2 {\bf u}_{i}^{\alpha}(x)|\leq C\Big(\frac{1}{\delta(x')}+\frac{|x'|}{\delta(x')^2}\Big),\quad x\in\Omega_{R},
	\end{align*}
	and for some $|z'|=R$,
	\begin{align}\label{improveof_p}
|p_i^\alpha-p_i^\alpha(z',0)|\leq \frac{C}{\varepsilon},	\quad|\nabla p_{i}^{\alpha}(x)|\leq C\Big(\frac{1}{\delta(x')}+\frac{|x'|}{\delta(x')^2}\Big),\quad x\in\Omega_{R}.
	\end{align}
\end{prop}
The proof of \eqref{improveof_p} will be given in Subsection \ref{sec4.4}.

For ${\bf u}_{1}^{4}$, since $\boldsymbol{\psi}_4$ does not include the component $x_3$, we directly choose
\begin{align*}
	{\bf v}_{1}^{4}=
	{\boldsymbol\psi}_{4}\Big(k(x)+\frac{1}{2}\Big),\quad\mbox{and}~\overline{p}_1^4=0,\quad \mbox{in}~\Omega_{2R}.
\end{align*}
Set $G(x)=-\frac{4x_{1}x_2}{\delta(x')}$, and for $i=1,2$,
\begin{equation*}
	\begin{split}
		F_i(x)=1-\frac{4x_{i}^2}{\delta(x')}-\frac{25x_{3}}{3}k(x),~
		H_i(x)=8x_i\left(\frac{2}{3}-\frac{|x'|^{2}}{\delta(x')}\right)k(x)-10x_ix_3k(x)^2.
	\end{split}
\end{equation*}
We construct ${\bf v}_{1}^{5}\in C^{3}(\Omega;\mathbb R^3)$, such that
\begin{align}\label{v15}
	{\bf v}_{1}^{5}={\boldsymbol\psi}_{5}\Big(k(x)+\frac{1}{2}\Big)
	+
	\frac{3}{5}\Big(k(x)^2-\frac{1}{4}\Big)(F_1(x),G(x),H_1(x))^{\mathrm T},\quad \mbox{in}~\Omega_{2R},
\end{align}
with $\|{\bf v}_{1}^{5}\|_{C^{3}(\Omega\setminus\Omega_{R})}\leq\,C$ and $\overline{p}_1^5\in C^{1}(\Omega)$, such that
\begin{equation*}
	\overline{p}_1^5=\frac{6\mu}{5}\frac{ x_1}{\delta(x')^2}+\frac{72\mu}{5}\frac{ x_1}{\delta(x')}\left(\frac{2}{3}-\frac{|x'|^{2}}{\delta(x')}\right)k(x)^{2},\quad\mbox{in}~\Omega_{2R},
\end{equation*}
with $\|\overline{p}_1^5\|_{C^{1}(\Omega\setminus\Omega_{R})}\leq C$. 
While, 
\begin{align*}
	{\bf v}_{1}^{6}={\boldsymbol\psi}_{6}\Big(k(x)+\frac{1}{2}\Big)
	+\frac{3}{5}\Big(k(x)^2-\frac{1}{4}\Big)(G(x),F_2(x),H_2(x))^{\mathrm T},\quad\mbox{in}~\Omega_{2R},
\end{align*}
with $\|{\bf v}_{1}^{6}\|_{C^{3}(\Omega\setminus\Omega_{R})}\leq\,C$, and
\begin{equation*}
	\overline{p}_1^6=\frac{6\mu}{5}\frac{ x_2}{\delta(x')^2}+\frac{72\mu}{5}\frac{ x_2}{\delta(x')}\left(\frac{2}{3}-\frac{|x'|^{2}}{\delta(x')}\right)k(x)^{2},\quad\mbox{in}~\Omega_{2R},
\end{equation*}
with $\|\overline{p}_1^6\|_{C^{1}(\Omega\setminus\Omega_{R})}\leq C$. Then 

\begin{prop}(An Improvement of \cite[Proposition 2.3]{LX2}\label{propu4563D})
	Let ${\bf u}_{i}^{\alpha}\in{C}^{2}(\Omega;\mathbb R^3),~p_{i}^{\alpha}\in{C}^{1}(\Omega)$ be the solution to \eqref{equ_v1}, $\alpha=4,5,6$. Then, for $x\in\Omega_{R}$,
	\begin{equation*}
		\|\nabla({\bf u}_{i}^{\alpha}-{\bf v}_{i}^{\alpha})\|_{L^{\infty}(\Omega_{\delta(x')/2}(x))}\leq\,C,
	\end{equation*}
and
	\begin{equation*}
		\|\nabla^2({\bf u}_{i}^{\alpha}-{\bf v}_{i}^{\alpha})\|_{L^{\infty}(\Omega_{\delta(x')/2}(x))}+\|\nabla (p_{i}^{\alpha}-\overline{p}_{i}^{\alpha})\|_{L^{\infty}(\Omega_{\delta(x')/2}(x))}\leq 
		\begin{cases}
			\frac{C}{\sqrt{\delta(x')}},&\alpha=4,\\
			\frac{C}{\delta(x')},&\alpha=5,6.
		\end{cases}
	\end{equation*}
	Consequently, for $x\in\Omega_{R}$,
	\begin{align*}
		|\nabla {\bf u}_{i}^{\alpha}(x)|\leq 
		\begin{cases}
			C\left(\frac{|x'|}{\delta(x')}+1\right),&\alpha=4,\\
			\frac{C}{\delta(x')},&\alpha=5,6,
		\end{cases}~
		|\nabla^2 {\bf u}_{i}^{\alpha}(x)|\leq
		\begin{cases}
			\frac{C}{\delta(x')},&\alpha=4,\\
			\frac{C}{\delta(x')^{2}},&\alpha=5,6,
		\end{cases}
	\end{align*}
	and for some $|z'|=R$,
	\begin{align*}
		|p_{i}^{\alpha}(x)-p_i^\alpha(z',0)|\leq
		\begin{cases}
			\frac{C}{\sqrt{\varepsilon}},&\alpha=4,\\
			\frac{C}{\varepsilon},&\alpha=5,6,
		\end{cases} ~
		|\nabla p_{i}^{\alpha}(x)|\leq
		\begin{cases}
			\frac{C}{\sqrt{\delta(x')}},&\alpha=4,\\
			\frac{C}{\delta(x')^{2}},&\alpha=5,6.
		\end{cases} 
	\end{align*}
\end{prop}

We directly use the following proposition from \cite{LX2}. For $\alpha\neq 3$, 
\begin{prop}\label{prop1.7}(\cite{LX2})
	Let ${\bf u}_{i}^{\alpha},{\bf u}_0\in{C}^{2}(\Omega;\mathbb R^3),~p_{i}^{\alpha},p_0\in{C}^{1}(\Omega)$ be the solution to \eqref{equ_v1} and \eqref{equ_u0}. If $\alpha\neq 3$, then we have
	\begin{equation*}
		\|\nabla^{k_1}{\bf u}_0\|_{L^{\infty}(\Omega)}+\|\nabla^{k_1}({\bf u}_{1}^{\alpha}+{\bf u}_{2}^{\alpha})\|_{L^{\infty}(\Omega)}\leq C,\quad k_1=1,2;
	\end{equation*}
	and
	\begin{equation*}
		\|\nabla^{k_2} p_0\|_{L^{\infty}(\Omega)}+\|\nabla^{k_2}(p_{1}^{\alpha}+p_{2}^{\alpha})\|_{L^{\infty}(\Omega)}\leq C,\quad k_2=0,1.
	\end{equation*}
\end{prop}

\subsection{Main Estimates for the Nonlinear Parts}\label{sec2.3}

For $\alpha=3$, although ${\bf u}_{1}^{3}+{\bf u}_{2}^{3}={\boldsymbol\psi}_{3}$ on $\partial D_1\cup\partial D_2$, there still being no potential difference on $\partial D_1$ and $\partial D_2$, like the assumptions in Proposition \ref{prop1.7}, the equation of ${\bf u}_1^3+{\bf u}_2^3$ is nonlinear. Nonetheless, we have the following control of ${\bf u}_1^3+{\bf u}_2^3$, including the undetermined differences $|C_1^\alpha-C_2^\alpha|$.

\begin{prop}\label{propuhe111}
	Let ${\bf u}_2^3, {\bf u}_1^3\in{C}^{2,\gamma}(\Omega;\mathbb R^2)$, $p_2^3, p_1^3\in{C}^{1,\gamma}(\Omega)$ be the solutions to \eqref{equ_v1} with $i=2,\alpha=3$, and \eqref{equ_v13}, respectively. Then
	\begin{equation*}
		\|\nabla({\bf u}_{1}^{3}+{\bf u}_{2}^{3})\|_{L^{\infty}(\Omega_{\delta(x')/2}(x))}\leq \frac{C|C_1^3-C_2^3|}{\sqrt{\delta(x')}}+C,
	\end{equation*}
	and 
	\begin{align*}
		&\|\nabla^2({\bf u}_{1}^{3}+{\bf u}_{2}^{3})(x)\|_{L^{\infty}(\Omega_{\delta(x')/2}(x'x))}+\|\nabla(p_1^3+p_2^3)\|_{L^{\infty}(\Omega_{\delta(x')/2}(x))}\\
	&\qquad\qquad\leq\,\frac{C\sum_{\alpha\neq4}|C_1^\alpha-C_2^\alpha|}{\delta(x')} +\frac{C|C_1^3-C_2^3||x'|}{\delta(x')^2}+\frac{C|C_1^4-C_2^4||x'|}{\delta(x')}+C.
	\end{align*}
\end{prop}

The following Proposition holds for $i=2$ (see \cite[Proposition 2.2]{LX2}), since $({{\bf u}}_2^3, p_2^3)$ satisfies the Stokes equations as well. But, the case $i=1$ is the most important thing we need to prove. 
\begin{prop}\label{propu133D}
	Let ${\bf u}_{i}^{3}\in{C}^{2}(\Omega;\mathbb R^3),~p_{i}^{3}\in{C}^{1}(\Omega)$ be the solution to \eqref{equ_v1}, then we have
	\begin{equation*}
		\|\nabla({\bf u}_{i}^{3}-{\bf v}_{i}^{3})\|_{L^{\infty}(\Omega_{\delta(x')/2}(x))}\leq\frac{C}{\sqrt{\delta(x')}},\quad x\in\Omega_{R},
	\end{equation*}
	and
	\begin{equation*}
		\|\nabla^2({\bf u}_{i}^{3}-{\bf v}_{i}^{3})\|_{L^{\infty}(\Omega_{\delta(x')/2}(x))}+\|\nabla (p_{i}^{3}-\overline{p}_{i}^{3})\|_{L^{\infty}(\Omega_{\delta(x')/2}(x))}\leq \frac{C}{\delta(x')^{3/2}},\quad x\in\Omega_{R},
	\end{equation*}
where ${\bf v}_{2}^{3}$ and $\overline{p}_{2}^{3}$ are constructed similarly as ${\bf v}_{1}^{3}$ and $\overline{p}_{1}^{3}$ below, only replacing $k(x)$ by $\widetilde{k}(x)=-\frac{x_3}{\delta(x')}$. Consequently, for $x\in\Omega_R$,
	\begin{align*}
		\frac{1}{C\delta(x')}\leq|\nabla {\bf u}_{i}^{3}(x)|\leq C\Big(\frac{1}{\delta(x')}+\frac{|x'|}{\delta(x')^2}\Big),~|\nabla^2 {\bf u}_{i}^{3}(x)|\leq C\left(\frac{1}{\delta(x')^{2}}+\frac{|x'|}{\delta(x')^{3}}\right),
	\end{align*}
	and for some $|z'|=R$,
	\begin{align*}
		|p_{i}^{3}(x)-p_i^{3}(z',0)|\leq\frac{C}{\varepsilon^2},~|\nabla p_{i}^{3}(x)|\leq C\left(\frac{1}{\delta(x')^{2}}+\frac{|x'|}{\delta(x')^{3}}\right).
	\end{align*}
\end{prop}

Next, we attempt to make advantage of the aforementioned estimates for $({\bf u}_{i}^{\alpha},p_{i}^{\alpha})_{(i,\alpha)\neq (1,3)}$ and $({\bf u}_{0},p_{0})$ to derive the estimates of $({\bf u}_1^3,p_{1}^{3})$, which satisfies the nonlinear equations \eqref{equ_u13}. We still choose, in $\Omega_{2R}$,
\begin{equation}\label{v133D}
	{\bf v}_{1}^{3}=\boldsymbol\psi_{3}\Big(k(x)+\frac{1}{2}\Big)+\frac{3}{\delta(x')}\left(\sum_{\alpha=1}^{2}x_{\alpha}\boldsymbol\psi_{\alpha}+2x_{3}\boldsymbol\psi_{3}\Big(\frac{|x'|^{2}}{\delta(x')}-\frac{1}{3}\Big)\right)\Big(k(x)^2-\frac{1}{4}\Big),
\end{equation}
and 
\begin{equation}\label{p133D}
	\overline{p}_1^3=-\frac{3}{2}\frac{\mu}{\delta(x')^2}+\frac{18\mu}{\delta(x')}\left(\frac{|x'|^{2}}{\delta(x')}-\frac{1}{3}\right)k(x)^{2},\quad\mbox{in}~\Omega_{2R},
\end{equation}
to prove Proposition \ref{propu133D} holds for $({\bf u}_1^3,p_{1}^{3})$ as well. This is one of the main contributions of this paper. 
To establish the estimates of $({\bf u}_1^3,p_{1}^{3})$, we set
\begin{align*}
	{\bf w}_{3}:={\bf u}_1^{3}-{\bf v}_{1}^{3},\quad\mbox{and}~~\quad~ q_{3}:=p_1^{3}-\overline{p}_{1}^{3},
\end{align*}
where ${\bf v}_{1}^{3}$ and $\overline{p}_{1}^{3}$ are defined in \eqref{v133D} and \eqref{p133D}. It follows from \eqref{equ_u13} that $({\bf w}_{3},q_3)$ verifies the following boundary value problem
\begin{equation}\label{w35}
	\left\{\begin{split}
		\mu\,\Delta {\bf w}_3&=\nabla q_3+(C_1^3)^2{\bf w}_3\cdot\nabla{\bf w}_3\\
		&\quad\quad\quad\quad+ C_1^3 (\widetilde{\bf v}\cdot\nabla {\bf w}_3 +{\bf w}_3\cdot\nabla\widetilde{\bf v}) -{\bf f},&\mathrm{in}&\;\Omega,\\
		\nabla\cdot{\bf{w}}_{3}&=0,&\mathrm{in}&\;\Omega_{2R},\\
		|\nabla\cdot{\bf{w}}_{3}|&\leq\,C,&\mathrm{in}&\;\Omega\setminus\Omega_{R},\\
		{\bf{w}}_{3}&=0,&\mathrm{on}&\;\partial \Omega,
\end{split}\right.
\end{equation}
where
\begin{equation}\label{f333}
{\bf f}:=\mu\,\Delta {\bf v}_1^3-C_1^3\,\widetilde{\bf v}\cdot \nabla {\bf v}_1^3-\nabla \overline{p}_1^3-\widetilde{\bf v}\cdot \nabla {\bf u}^{\#3},\quad\mbox{and}\quad~~~~
\widetilde{\bf v}:=C_1^3{\bf v}_1^3+ {\bf u}^{\#3},
\end{equation}
are known terms, based on the estimates presented in Propositions \ref{lemma11}-\ref{prop1.7} and Proposition \ref{propu133D} for $i=2$. The proof of Proposition \ref{propu133D} for $i=1$ and some preliminary results for the Navier-Stokes equations are given in Section \ref{sec3} and \ref{sec4}.

Based on the above pointwise estimates of $\nabla{\bf u}_{i}^{\alpha}$ and $p_{i}^{\alpha}$, by solving a linear system of $C_{i}^{\alpha}$ derived from  \eqref{sto-2}, we have

\begin{prop}\label{lemCialpha3D}
	Let $C_{i}^{\alpha}$ be defined in \eqref{udecom}. Then
	\begin{equation*}
		|C_1^{\alpha}-C_2^{\alpha}|\leq \frac{C}{|\log\varepsilon|}, ~\alpha=1,2,5,6,\quad |C_1^3-C_2^3|\leq C\varepsilon,~\mbox{and}~|C_1^4-C_2^4|\leq C.
	\end{equation*}
\end{prop}
The proof is given in Subsection \ref{subsec43}.

\subsection{The Completion of the Proof of Theorem \ref{mainthm}}
We are now in a position to complete the proof of Theorem \ref{mainthm}.

\begin{proof}[Proof of Theorem \ref{mainthm}]
	We  recombine \eqref{udecom} as
	\begin{equation*}
		{\bf u}=\sum_{\alpha=1}^{6}(C_{1}^{\alpha}-C_{2}^{\alpha}){\bf u}_{1}^{\alpha}
		+ {\bf u}_{b},\quad\mbox{and }~
		p=\sum_{\alpha=1}^{6}(C_{1}^{\alpha}-C_{2}^{\alpha})p_{1}^{\alpha}+p_{b},\quad\mbox{in}~\Omega,
	\end{equation*}
	where
	\begin{equation*}
		{\bf u}_{b}:=\sum_{\alpha=1}^{6}C_{2}^{\alpha}({\bf u}_{1}^{\alpha}+{\bf u}_{2}^{\alpha})+{\bf u}_{0},\quad p_{b}:=\sum_{\alpha=1}^{6}C_{2}^{\alpha}(p_{1}^{\alpha}+p_{2}^{\alpha})+p_{0}.
	\end{equation*}
	
$(i)$ By virtue of  Propositions \ref{propu113D1}--\ref{lemCialpha3D}, 
	\begin{align*}
		|\nabla{\bf u}(x)|\leq&\,\sum_{\alpha=1}^{6}\left|(C_{1}^{\alpha}-C_{2}^{\alpha}\right)\nabla{\bf u}_{1}^{\alpha}(x)|+C\nonumber\\
		\leq&\,\left|(C_{1}^{3}-C_{2}^{3}\right)\nabla{\bf u}_{1}^{3}(x)|+\sum_{\alpha=1,2,5,6}\left|(C_{1}^{\alpha}-C_{2}^{\alpha}\right)\nabla{\bf u}_{1}^{\alpha}(x)|+C|\nabla{\bf u}_{1}^{4}(x)|+C\nonumber\\
		\leq&\,C\varepsilon\Big(\frac{1}{\delta(x')}+\frac{|x'|}{\delta(x')^2}\Big)+\frac{C}{|\log\varepsilon|\delta(x')}+C\left(\frac{|x'|}{\delta(x')}+1\right)\nonumber\\
		\leq&\, \frac{C(1+|\log\varepsilon||x'|)}{|\log\varepsilon|(\varepsilon+|x'|^2)}.
	\end{align*}
For $p$, by using Proposition \ref{lemma11} and Proposition \ref{prop1.7}, we have, for $|z'|=R$,
	\begin{align*}
	|p_b(x)-p_b(z',0)|\leq C\sum_{\alpha=1}^6|\nabla (p_1^\alpha+p_2^\alpha)|+C|\nabla p_0|\leq \frac{C}{|\log\varepsilon|\varepsilon}.
\end{align*}
So that, by virtue of Propositions \ref{propu113D1}--\ref{lemCialpha3D} again,
\begin{align*}
	|p(x)-p(z',0)|
	\leq&\, C\sum_{\alpha=1}^{6}|C_{1}^{\alpha}-C_{2}^{\alpha}|\Big(\delta(x')|\nabla p_{1}^{\alpha}(x)|+|p_{1}^{\alpha}(x',0)-p_{1}^{\alpha}(z',0)|\Big)+|p_b(x)-p_b(z',0)|.\\
	\leq&\,\frac{C}{|\log\varepsilon|\varepsilon}.
\end{align*}

$(ii)$ For the second order derivatives, 
	\begin{align*}
	|\nabla^2{\bf u}(x)|
	\leq&\,\sum_{\alpha=1,2}\left|(C_{1}^{\alpha}-C_{2}^{\alpha}\right)\nabla^2{\bf u}_{1}^{\alpha}(x)|+\left|(C_{1}^{3}-C_{2}^{3}\right)\nabla^2{\bf u}_{1}^{3}(x)|+C|\nabla^2{\bf u}_{1}^{4}(x)|\nonumber\\
	&\quad\quad\quad+\sum_{\alpha=5,6}\left|(C_{1}^{\alpha}-C_{2}^{\alpha}\right)\nabla^2{\bf u}_{1}^{\alpha}(x)|+|\nabla^2{\bf u}_{b}(x)|\nonumber\\
	\leq&\,\frac{C}{|\log\varepsilon|}\Big(\frac{1}{\delta(x')}+\frac{|x'|}{\delta(x')^2}\Big)+C\varepsilon\left(\frac{1}{\delta(x')^{2}}+\frac{|x'|}{\delta(x')^{3}}\right)+\frac{C}{\delta(x')}+\frac{C}{|\log\varepsilon|\delta(x')^2}\nonumber\\
	\leq&\,\frac{C(1+|\log\varepsilon||x'|)}{|\log\varepsilon|(\varepsilon+|x'|^2)^{2}},
\end{align*}
	and 
	\begin{align*}
|\nabla p(x)|\leq&\,\sum_{\alpha=1,2}\left|(C_{1}^{\alpha}-C_{2}^{\alpha}\right)\nabla p_{1}^{\alpha}(x)|+\left|(C_{1}^{3}-C_{2}^{3}\right)\nabla p_{1}^{3}(x)|+C|\nabla p_{1}^{4}(x)|\nonumber\\
		&\quad\quad+\sum_{\alpha=5,6}\left|(C_{1}^{\alpha}-C_{2}^{\alpha}\right)\nabla p_{1}^{\alpha}(x)|+|\nabla p_{b}(x)|\nonumber\\
		\leq&\,\frac{C}{|\log\varepsilon|}\left(\frac{1}{\delta(x')}+\frac{|x'|}{\delta(x')^{2}}\right)+\left(\frac{C\epsilon}{\delta(x')^{2}}+\frac{C\varepsilon|x'|}{\delta(x')^{3}}\right)+\frac{C}{\sqrt{\delta(x')}}+\frac{C}{|\log\varepsilon|\delta(x')^2}\nonumber\\
		\leq&\, \frac{C(1+|\log\varepsilon||x'|)}{|\log\varepsilon|(\varepsilon+|x'|^2)^{2}}.
	\end{align*}
Thus, the proof of Theorem \ref{mainthm} is completed.
\end{proof}

\begin{proof}[On Remark \ref{mainthmstokes}]
For the Stokes equations, by \cite[Propositions 2.4, 2.5]{LX2}, we have
	\begin{align*}
		|C_1^{\alpha}-C_2^{\alpha}|\leq \frac{C}{|\log\varepsilon|}, ~\alpha=1,2,5,6,\quad |C_1^3-C_2^3|\leq C\varepsilon,~\mbox{and}~|C_1^4-C_2^4|\leq C,
	\end{align*}
	and
	$$|C_2^\alpha|,~\|\nabla (p_1^\beta+p_2^\beta)\|_{L^\infty(\Omega)}\leq C,~~\beta=1,2,\dots 6.$$
For $z=(z',0)\in\Omega$ with $|z'|=R$, by mean value theorem, 
	\begin{align*}
		|p_b(x)-p_b(z',0)|\leq C\sum_{\beta=1}^6\|\nabla (p_1^\beta+p_2^\beta)\|_{L^\infty(\Omega_{R})}+C\|\nabla p_0\|_{L^\infty(\Omega_{R})}\leq C.
	\end{align*}
Hence, using mean value theorem again, we derive
	\begin{align*}
		|p(x)- p(z',0)|
		\leq&\,\sum_{\alpha=1}^{6}|(C_{1}^{\alpha}- C_{2}^{\alpha})|\Big(|\delta(x')\|\nabla p_{1}^{\alpha}\|_{L^\infty(\Omega_{R})}+|p_1^\alpha(x',0)- p_1^\alpha(z',0)|\Big)+ C\\
		\leq&\,\frac{C}{|\log\varepsilon|\varepsilon}.
	\end{align*}
We finish the proof on Remark \ref{mainthmstokes}.
\end{proof}

\begin{remark}\label{rem28}
Since ${\bf u}_i^\alpha-{\bf v}_i^\alpha=0$ on $\partial D_i$, it follows from Propositions \ref{propu113D1} and \ref{propu4563D}, Proposition \ref{propu133D} for $({\bf u}_2^3,p_2^3)$ and the mean-value theorem that
	\begin{align}\label{zydfzgs11}
		|{\bf u}_i^\alpha(x)-{\bf v}_i^\alpha(x)|\leq C\sqrt{\delta(x')}, \quad\mbox{for}~ (i,\alpha)\neq (1,3),~x\in \Omega_{R}.
	\end{align}
While, it is easy to see from \eqref{v133D} that  
\begin{align}\label{boundofv13}
|{\bf v}_{1}^{3}(x)|\leq\frac{C|x'|}{\delta(x')}+C,\quad\mbox{for}~x\in \Omega_{R},
\end{align}
 which implies that it may blow up near the origin.
\end{remark}

\section{Some Technical Results on the Navier-Stokes Equations}\label{sec3}

In this section we prove some preliminary results for the Navier-Stokes equations of type
\begin{align}\label{w36}
	\begin{cases}
		\mu\,\Delta {\bf w}=\nabla q+{\bf w}\cdot\nabla{\bf w}+\widetilde{\bf v}\cdot\nabla {\bf w} +{\bf w}\cdot\nabla\widetilde{\bf v}-{\bf f},&\mathrm{in}\;\Omega,\\
		\nabla\cdot{\bf{w}}=0,&\mathrm{in}\;\Omega_{2R},\\
		|\nabla\cdot{\bf{w}}|\leq\,C,&\mathrm{in}\;\Omega\setminus\Omega_{R},\\
		{\bf{w}}=0,&\mathrm{on}\;\partial \Omega,
	\end{cases}
\end{align}
where $\widetilde{\bf v}$ and ${\bf f}$ are known terms, to make some preparations for applying the iteration technique in the next section.

\subsection{Boundedness of the Total Energy}

We have the following boundedness of the total energy of ${\bf w}$ in $\Omega$.
\begin{lemma}\label{lemmaenergy}
	 Assume that
	\begin{equation}\label{estv113D3}
		|\widetilde{\bf v}(x)|\leq\frac{C}{\sqrt{\delta(x')}},\quad |\nabla\widetilde{\bf v}(x)|\leq\frac{C}{\delta(x')^{3/2}}, ~\text{in}~\Omega_{R},
	\end{equation}
and
\begin{equation}\label{estv113d1}
	\|\widetilde{\bf v}\|_{C^{2}(\Omega\setminus\Omega_R)}+\|{\bf f}\|_{L^{\infty}(\Omega\setminus\Omega_R)}+\|\nabla q\|_{C^{0}(\Omega\setminus\Omega_R)}\leq\,C.
\end{equation}
Let $({\bf w},q)$ be the solution to \eqref{w36}. If further
	\begin{equation}\label{int-fw}
		\Big| \int_{\Omega_{R}}\sum_{j=1}^{d}{\bf f}^{(j)}{\bf w}^{(j)}\mathrm{d}x\Big|\leq\,C\left(\int_{\Omega}|\nabla {\bf w}|^2\mathrm{d}x\right)^{1/2},
	\end{equation}
and, for a given divergence-free function ${\bf u}$ in $\Omega$, 
\begin{align}\label{jstj}
|{\bf w}(x)-{\bf u}(x)|\leq \frac{C}{\sqrt{\delta(x')}},\quad\quad~~\mbox{in}~~\Omega_{R},
\end{align}
then for sufficiently small $0<\epsilon<1/2$, we have
	\begin{align}\label{est111}
		\int_{\Omega}|\nabla {\bf w}|^{2}\mathrm{d}x\leq C.
	\end{align}
\end{lemma}
\begin{proof}
Since $({\bf w},q)$ is the solution to \eqref{w36}, it also verifies
\begin{align}\label{fzfc1}
	-\mu\,\Delta {\bf w}+\nabla (q-q(z',0))+{\bf w}\cdot\nabla{\bf w}+\widetilde{\bf v}\cdot\nabla {\bf w} +{\bf w}\cdot\nabla\widetilde{\bf v}={\bf f},
\end{align}
for some $z'$ with $|z'|=R$. Multiplying equation \eqref{fzfc1} by ${\bf w}$, integrating by parts, and using the divergence-free condition, we deduce
	\begin{align}\label{dsjzygj1}
		\mu\int_{\Omega}|\nabla {\bf w}|^2+\int_{\Omega}{\bf w}\cdot \nabla {\bf w} \cdot {\bf w}&+\int_{\Omega}\widetilde{\bf v}\cdot\nabla {\bf w}\cdot{\bf w}+\int_{\Omega}{\bf w}\cdot\nabla\widetilde{\bf v}\cdot{\bf w}\nonumber\\
		&+\int_{\Omega\setminus\Omega_R}(q-q(z',0))\nabla\cdot{\bf w}=\int_{\Omega} {\bf f}\cdot {\bf w}.
	\end{align}

For the second term on the left hand side of \eqref{dsjzygj1}, since ${\bf w}=0$ on $\partial \Omega$ and $\nabla\cdot{\bf u}=0$ in $\Omega$, it follows from \cite[Lemma IX.2.1]{GaldiBook}  that 
	\begin{align*}
     \int_{\Omega} {\bf u}\cdot \nabla {\bf w} \cdot {\bf w}=0.
	\end{align*}
On the other hand, by means of \eqref{jstj}, the H{\"o}lder's inequality and the Hardy inequality, we deduce
    \begin{align*}
    	 \Big| \int_{\Omega_R} ({\bf w}-{\bf u})\cdot \nabla {\bf w} \cdot {\bf w}\Big|\leq C\Big\|\sqrt{\delta(x')}\cdot\nabla {\bf w}\Big\|_{L^2{(\Omega_{R})}}\Big\|\frac{{\bf w}}{x_{d}}\Big\|_{L^2{(\Omega_{R})}}\leq C \sqrt{\varepsilon+R^2}\int_{\Omega}|\nabla{\bf w}|^2.
    \end{align*}
Hence, 
\begin{align}\label{dsjzygj2}
\Big|\int_{\Omega}{\bf w}\cdot \nabla {\bf w} \cdot {\bf w}\Big|\leq C \sqrt{\varepsilon+R^2}\int_{\Omega}|\nabla{\bf w}|^2+C.
\end{align}

For the third term in \eqref{dsjzygj1}, by using assumption \eqref{estv113D3}, the H{\"o}lder's inequality and the Hardy inequality, we have
  \begin{align*}
  	\Big|\int_{\Omega_R} \widetilde{\bf v}\cdot \nabla {\bf w}\cdot{\bf w} \Big|\leq&\, C\|\delta(x')\widetilde{\bf v}\cdot\nabla {\bf w}\|_{L^2{(\Omega_{R})}}\bigg\|\frac{{\bf w}}{x_{d}}\bigg\|_{L^2{(\Omega_{R})}}\leq C \sqrt{\varepsilon+R^2}\int_{\Omega}|\nabla{\bf w}|^2.
  \end{align*}
Since ${\bf w}=0$ on $\partial D_1\cup\partial D_2$, then 
\begin{align*}
	|{\bf w}(x',x_{d})|^2\leq \Big|\int_{-\frac{\varepsilon}{2}-\frac{1}{2}|x'|^2}^{\frac{\varepsilon}{2}+\frac{1}{2}|x'|^2}\partial_{x_{d}}{\bf w}(x',x_{d})\mathrm{d}x_{d}\Big|^2\leq \delta(x')\int_{-\frac{\varepsilon}{2}-\frac{1}{2}|x'|^2}^{\frac{\varepsilon}{2}+\frac{1}{2}|x'|^2}|\nabla {\bf w}|^2 \mathrm{d}x_{d}.
\end{align*}
Combining with \eqref{estv113D3} yields
   \begin{align*}
   \Big|\int_{\Omega_{R}}{\bf w}\cdot\nabla \widetilde{\bf v}\cdot{\bf w}\Big|\leq \int_{\Omega_{R}} \frac{C}{\delta(x')^{3/2}}|{\bf w}|^2\leq C\sqrt{\varepsilon+R^2}\int_{\Omega_{R}}|\nabla{\bf w}|^2.
   \end{align*} 
Then, after shrinking $\varepsilon$ and $R$ if necessary, yields
  	\begin{align}\label{gj11}
  	\Big|\int_{\Omega}\widetilde{\bf v}\cdot\nabla {\bf w}\cdot{\bf w}+{\bf w}\cdot\nabla\widetilde{\bf v}\cdot{\bf w}\Big|\leq \frac{\mu}{2} \int_{\Omega}|\nabla{\bf w}|^2+C.
  \end{align}

For the last term, by means of assumptions \eqref{estv113d1}
\begin{align}\label{gj12}
\Big|\int_{\Omega\setminus\Omega_R}(q-q(z',0))\nabla\cdot{\bf w}\Big|\leq C\|\nabla q\|_{L^{\infty}(\Omega\setminus\Omega_R)}\leq C.
\end{align}
Finally, by means of assumptions \eqref{estv113d1} and \eqref{int-fw}, and the Poincar{\'e} inequality, we deduce
	\begin{align}\label{gj22}
	\Big|\int_{\Omega}{\bf f}\cdot{\bf w}\Big|\leq C\Big|\int_{\Omega_R}{\bf f}\cdot{\bf w}\Big|+C\leq C \left(\int_{\Omega}|\nabla {\bf w}|^2\right)^{1/2}+C.
\end{align}
Thus, combining \eqref{dsjzygj1}-\eqref{gj22} leads to \eqref{est111}. The proof of Lemma \ref{lemmaenergy} is finished.
\end{proof}

\subsection{A Modified Iteration Formula}
The following Caccioppoli-type inequality is a key step to build an adapted version of iteration formula, used in \cite{LX2}.
\begin{lemma}
	(Caccioppoli-type Inequality) Let $({\bf w},q)$ be the solution to \eqref{w36}. Under the assumptions in Lemma \ref{lemmaenergy}, then for $|z'|<R$ and $0<t<s\leq R$, 	\begin{align}\label{nsiterating1}
\int_{\Omega_{t}(z')}|\nabla {\bf w}|^{2}
\leq&\,\Big(\frac{1}{4}+\frac{C\delta(z')^2}{(s-t)^{2}}+\frac{1}{8}\frac{\delta(z')}{s-t}\Big)\int_{\Omega_{s}(z')}|\nabla {\bf w}|^2+C\Big((s-t)^{2}+\delta(z')^{2}\Big)\int_{\Omega_{s}(z')}|{\bf f}|^{2}.
 \end{align}
\end{lemma}    	
\begin{proof}
For $0<t<s\leq \sqrt{\delta(z')}$, since ${\bf w}=0$ on $\partial D_1\cup\partial D_2$, we have the following Poincar\'e inequality
	\begin{align}\label{estwDwl2}
		\int_{\Omega_{s}(z')}|{\bf w}|^2\mathrm{d}x\leq
		C\delta(z')^2\int_{\Omega_{s}(z')}|\nabla {\bf w}|^2\mathrm{d}x.
	\end{align}
Let $\eta$ be a smooth cutoff function satisfying $\eta(x')=1$ if $|x'-z'|<t$, $\eta(x')=0$ if $|x'-z'|>s$, $0\leqslant\eta(x')\leqslant1$,
	and $|\nabla\eta|\leq \frac{2}{s-t}$. Denote $q_{s;z'}=\frac{1}{|\Omega_{s}(z')|}\int_{\Omega_{s}(z')}q$. 
Multiplying the equation 
	$$	-\mu\,\Delta {\bf w}+{\bf w}\cdot\nabla{\bf w}+\widetilde{\bf v}\cdot\nabla {\bf w} +{\bf w}\cdot\nabla\widetilde{\bf v}+\nabla(q-q_{s;z'})={\bf f}, \quad\mbox{in}~\Omega_{s}(z'),$$
by $\eta^{2}{\bf w}$, and integrating by parts, leads to
	\begin{align}\label{nsenergynarr}
\int_{\Omega_{s}(z')}\eta^{2}|\nabla {\bf w}|^{2}
=&-\int_{\Omega_{s}(z')}({\bf w}\nabla {\bf w})\cdot\nabla\eta^{2}+\frac{1}{\mu}\int_{\Omega_{s}(z')} \eta^{2}{\bf f}\cdot{\bf w}+\frac{1}{\mu}\int_{\Omega_{s}(z')}\nabla\cdot(\eta^{2}{\bf w})\left(q-q_{s;z'}\right)\nonumber\\
		&-\frac{1}{\mu}\int_{\Omega_{s}(z')}({\bf w}\cdot\nabla {\bf w})\cdot\eta^{2}{\bf w}-\frac{1}{\mu}\int_{\Omega_{s}(z')}(\widetilde{\bf v}\cdot\nabla {\bf w})\cdot\eta^{2}{\bf w}-\frac{1}{\mu}\int_{\Omega_{s}(z')}({\bf w}\cdot\nabla \widetilde{\bf v})\cdot\eta^{2}{\bf w}\nonumber\\
		:=&\,\mbox{I}_1+\mbox{I}_2+\mbox{I}_3+\mbox{I}_4+\mbox{I}_5+\mbox{I}_6.
	\end{align}

By using the argument as in the proof of \cite[Lemma 3.10]{LX} and \eqref{estwDwl2}, we derive
	\begin{align}\label{local1}
		|\mbox{I}_1+\mbox{I}_2+\mbox{I}_3|\leq&\, \Big(\frac{1}{16}+\frac{C\delta(z')^2}{(s-t)^{2}}\Big)\int_{\Omega_{s}(z')}|\nabla {\bf w}|^2+C\Big((s-t)^{2}+\delta(z')^{2}\Big)\int_{\Omega_{s}(z')}|{\bf f}|^2.
	\end{align}
Using \eqref{estwDwl2} again,
\begin{align*}
\|\eta^2{\bf w}\|_{W^{1,2}(\Omega_{s}(z'))}\leq&\,C\Big(\|{\bf w}\|_{L^{2}(\Omega_{s}(z'))}+\frac{2}{s-t}\|{\bf w}\|_{L^{2}(\Omega_{s}(z'))}+\|\nabla {\bf w}\|_{L^2(\Omega_{s}(z'))}\Big)\\
\leq&\,C\Big(\frac{1}{2}+\frac{\delta(z')}{s-t}\Big)\|\nabla{\bf w}\|_{L^2(\Omega_{s}(z'))}.
\end{align*}
By virtue of \cite[Lemma IX.1.1]{GaldiBook},
	\begin{align*}
		|\mbox{I}_4|=\Big|\int_{\Omega_{s}(z')} {\bf w}\cdot\nabla{\bf w}\cdot \eta^{2}{\bf w}\Big|
		\leq&\, \frac{2\sqrt{2}}{3}|\Omega_{s}(z')|^{1 / 6}\|{\bf w}\|^2_{W^{1,2}(\Omega_{s}(z'))}\|\eta^2{\bf w}\|_{W^{1,2}(\Omega_{s}(z'))}.
	\end{align*}
Hence
$$|\mbox{I}_4|\leq\, C\delta(z')^{1/3}\|\nabla{\bf w}\|_{L^2(\Omega)}\Big(\frac{1}{2}+\frac{\delta(z')}{s-t}\Big)\|\nabla{\bf w}\|^{2}_{L^2(\Omega_{s}(z'))}\leq\, \frac{1}{8}\Big(\frac{1}{2}+\frac{\delta(z')}{s-t}\Big) \int_{\Omega_{s}(z')} |\nabla{\bf w}|^2,$$
here we use \eqref{est111}, and shrink $\varepsilon$ and $R$ if necessary (for example, we choose sufficiently small $\epsilon$ and $R$ such that $C\delta(z')^{1/3}<1/8$ for $|z'|<R$, then fix it).
 
By the Cauchy inequality, and using \eqref{estv113D3} and \eqref{estwDwl2}, 
	\begin{align}\label{local3}
		|\mbox{I}_5|=&\,\Big|\int_{\Omega_{s}(z')}\widetilde{\bf v}\cdot\nabla{\bf w}\cdot \eta^{2}{\bf w}\Big|\leq \frac{1}{32} \int_{\Omega_{s}(z')} |\nabla {\bf w}|^2 + C\int_{\Omega_{s}(z')} |\widetilde{\bf v}|^2|\eta^2 {\bf w}|^2\nonumber\\
		\leq&\, \frac{1}{32} \int_{\Omega_{s}(z')} |\nabla {\bf w}|^2 + C\delta(z')\int_{\Omega_{s}(z')} |\nabla{\bf w}|^2
		\leq\,\frac{1}{16} \int_{\Omega_{s}(z')} |\nabla {\bf w}|^2,
	\end{align}
if $C\delta(z')<1/32$. By virtue of \eqref{estv113D3} and \eqref{estwDwl2} again, 
	\begin{align}\label{local4}
		|\mbox{I}_6|=\Big|\int_{\Omega_{s}(z')}{\bf w}\cdot\nabla\widetilde{\bf v}\cdot \eta^{2}{\bf w}\Big|\leq C\sqrt{\delta(z')}\int_{\Omega_{s}(z')} |\nabla{\bf w}|^2\leq\frac{1}{16}\int_{\Omega_{s}(z')} |\nabla{\bf w}|^2,
	\end{align}
if $C\sqrt{\delta(z')}<1/16$. Combining \eqref{local1}--\eqref{local4} with \eqref{nsenergynarr} yields \eqref{nsiterating1}.
\end{proof} 

\subsection{$W^{2,\infty}$ Estimates}
We recall a well known $L^{q}$-estimates for the Stokes system in a bounded domain with partially vanishing boundary data, see \cite[Theorem IV.5.1]{GaldiBook}.

\begin{theorem}\label{thmWmq}
	Let $\Omega$ be an arbitrary domain in $\mathbb R^{d}$, $d\geq2$, with a boundary portion $\sigma$ of class $C^{m+2}$, $m\geq0$. Let $\Omega_0$ be any bounded subdomain of $\Omega$ with $\partial\Omega_0\cap\partial\Omega=\sigma$. Further, let
	\begin{align*}
		{\bf u}\in W^{1,q}(\Omega_0), \quad p\in L^q(\Omega_0),\quad 1<q<\infty,
	\end{align*}
	be such that
	\begin{align*}
		(\nabla{\bf u},\nabla{\boldsymbol\psi})&=-\langle{\bf f},{\boldsymbol\psi}\rangle+(p,\nabla\cdot{\boldsymbol\psi}),\quad\mbox{for~all}~{\boldsymbol\psi}\in C_0^\infty(\Omega_0),\\
		({\bf u},\nabla\varphi)&=0,\quad \mbox{for~all}~\varphi\in C_0^\infty(\Omega_0),\\
		{\bf u}&=0,\quad \mbox{at}~\sigma.
	\end{align*}
	Then, if ${\bf f}\in W^{m,q}(\Omega_0)$, we have
	$${\bf u}\in W^{m+2,q}(\Omega'),\quad p\in W^{m+1,q}(\Omega'),$$
	for any $\Omega'$ satisfying 
	\begin{enumerate}
		\item $\Omega'\subset\Omega$,
		\item $\partial\Omega'\cap\partial\Omega$ is a strictly interior subregion of $\sigma$.
	\end{enumerate}
	Finally, the following estimate holds
	\begin{align*}
		\|{\bf u}\|_{W^{m+2,q}(\Omega')}+\|p\|_{W^{m+1,q}(\Omega')}\leq C\left(\|{\bf f}\|_{W^{m,q}(\Omega_0)}+\|{\bf u}\|_{W^{1,q}(\Omega_0)}+\|p\|_{L^{q}(\Omega_0)}\right),
	\end{align*}
	where $C=C(d,m,q,\Omega',\Omega_0)$.
\end{theorem}

The following lemma is needed, which is from \cite[Chapter V, Lemma 3.1]{M1}.
\begin{lemma}\label{itLemma}
 Let $f(t)$ be a nonnegative bounded function defined in $[\tau_{0},\tau_{1}]$, $\tau_{0}\geq 0$. Suppose that for $\tau_{0}\leq t<s\leq\tau_{1}$, we have
 $$f(t)\leq [A(s-t)^{-a}+B]+\theta\,f(s),$$
 where $A,B, a,\theta$ are nonnegative constants with $0\leq \theta<1$. Then for all  $\tau_{0}\leq \rho<\tau_{1}$ there holds
 $$f(\rho)\leq\,c\,[A(R-\rho)^{-a}+B],$$
 where $c$ is a constant depending on $a$ and $\theta$.
\end{lemma}

By a rescaling argument, we have the following Proposition.
\begin{prop}\label{nsdpro}
Let $({\bf w},q)$ be the solution to \eqref{w36}. Under the assumptions in Lemma \ref{lemmaenergy}, if additionally for $z\in \Omega_{R}$,  
	\begin{equation}\label{fzzyj1}
		\int_{\Omega_{\delta(z')}(z)} |\nabla {\bf w}|^2\leq C \delta(z')^2,
	\end{equation}
then
	\begin{equation}\label{nsW2pstokes}
		\|\nabla {\bf w}\|_{L^{\infty}(\Omega_{\delta(z')/2}(z))}\leq\,
		\,C\left(\delta(z')^{-d/2}\|\nabla {\bf w}\|_{L^{2}(\Omega_{\delta(z')}(z))}+\delta(z')\|{\bf f}\|_{L^{\infty}(\Omega_{\delta(z')}(z))} \right),
	\end{equation}
  and
  \begin{align}\label{Wmpstokes}
  	&\|\nabla^{2}{\bf w}\|_{L^{\infty}(\Omega_{\delta(z')/2}(z))}+\|\nabla q\|_{L^{\infty}(\Omega_{\delta(z')/2}(z))}\nonumber\\
  	\leq&\, C\left(\delta(z')^{-1-{d}/{2}}\|\nabla {\bf w}\|_{L^{2}(\Omega_{\delta(z')}(z))}+\|{\bf f}\|_{L^{\infty}(\Omega_{\delta(z')}(z))} +\delta(z')\|\nabla{\bf f}\|_{L^{\infty}(\Omega_{\delta(z')}(z))} \right).
  \end{align}
\end{prop}

\begin{proof} 
For $z=(z', z_{d})\in \Omega_{R}$, we take the change of variables as before
	\begin{equation*}
		\left\{
		\begin{aligned}
			&x'-z'=\delta(z') y',\\
			&x_d=\delta(z') y_{d},
		\end{aligned}
		\right.
	\end{equation*}
	which transform $\Omega_{\delta(z')}(z)$ into a nearly unit size domain 
	$Q_{1}$, where 
	\begin{equation*}
		Q_{r}=\left\{y\in\mathbb{R}^{d}\Big|~|y_{d}|
		<\frac{\varepsilon}{2\delta(z')}+\frac{(z'+\delta(z')\,y')^2}{2\delta(z')},~|y'|<r\right\},
	\end{equation*}
with top and bottom boundaries denoted by $\hat\Gamma_{1}^{+}$ and $\hat\Gamma_{1}^{-}$. For $y\in{Q}_{1}$, we set
	\begin{align*}
		\mathcal{W}(y', y_{d})&:={\bf w}(z'+\delta(z')\,y',\delta(z')\,y_{d}),\quad~
		\mathcal{G}(y', y_{d}):=\delta(z') q(z'+\delta(z')\,y',\delta(z')\,y_{d}),\\
		\mathcal{F}(y', y_{d})&:={\bf f}(z'+\delta(z')\,y',\delta(z')\,y_{d}),\quad
		\widetilde{\mathcal{V}}(y', y_{d}):=\widetilde{\bf v}(z'+\delta(z')\,y',\delta(z')\,y_{d}),
	\end{align*}
 and  $\mathcal{G}_1(y', y_{d}):=\frac{1}{|Q_1|}\int_{Q_1}	\mathcal{G}$. It follows from \eqref{w36} that
	\begin{align*}
		\begin{cases}
			-\mu\Delta\mathcal{W}+\nabla(\mathcal{G}-\mathcal{G}_1)=\mathcal{F}-\mathcal{M},
			& \mathrm{in}\,Q_1,\\
			\nabla \cdot \mathcal{W} =0,&\mathrm{in}\,Q_1,\\
			\mathcal{W}=0,&\mathrm{on}\, \hat{\Gamma}^+\cup\hat\Gamma^{-},
		\end{cases}
	\end{align*}
where $\mathcal{M}:=\delta(z')(	\widetilde{\mathcal{V}}\cdot\nabla\mathcal{W}+\mathcal{W}\cdot\nabla\widetilde{\mathcal{V}}+\mathcal{W}\cdot\nabla\mathcal{W})$. We estimate the terms in $\mathcal{M}$. By virtue of \eqref{estv113D3}, 
	\begin{align*}
		\|\delta(z'){\widetilde{\mathcal{V}}}\cdot\nabla\mathcal{W}\|_{L^4(Q_1)}\leq C\|\delta(z')\widetilde{\mathcal{V}}\|_{L^\infty(Q_1)}\|\nabla\mathcal{W}\|_{L^4(Q_1)}\leq C\|\nabla\mathcal{W}\|_{L^4(Q_1)},
	\end{align*}
and
\begin{align*}
		\|\delta(z')\mathcal{W}\cdot\nabla\widetilde{\mathcal{V}}\|_{L^4(Q_1)}\leq C \|\delta(z')\nabla\widetilde{\mathcal{V}}\|_{L^{\infty}(Q_1)}\|\mathcal{W}\|_{L^{4}(Q_1)}\leq C\|\mathcal{W}\|_{L^{4}(Q_1)}.
	\end{align*}
For the last term, using the Poincar\'e inequality $ \|\mathcal{W}\|_{W^{1,2}(Q_1)}\leq\,C \|\nabla\mathcal{W}\|_{L^2(Q_1)}$, it follows from the Sobolev embedding theorem, the H\"{o}lder’s inequality that
	\begin{align*}
		&\|\delta(z')\mathcal{W}\cdot\nabla\mathcal{W}\|_{L^4(Q_1)}\leq\, C\delta(z') \|\mathcal{W}\|_{L^6(Q_1)}\|\nabla\mathcal{W}\|_{L^{12}(Q_1)}\nonumber\\
		\leq&\, C\delta(z') \|\mathcal{W}\|_{W^{1,2}(Q_1)}\|\nabla\mathcal{W}\|_{W^{1,4}(Q_1)}
		\leq\, C\delta(z') \|\nabla\mathcal{W}\|_{L^2(Q_1)}\|\mathcal{W}\|_{W^{2,4}(Q_1)}.
	\end{align*}
By means of \eqref{fzzyj1},
$$ \delta(z')\|\nabla\mathcal{W}\|_{L^{2}(Q_1)}=\delta(z')^{1/2}\|\nabla{\bf w}\|_{L^{2}\big(\Omega_{\delta(z')}(z)\big)}\leq C\delta(z')^{3/2},$$
thus
	\begin{align*}
	\|\mathcal{M}\|_{L^{4}(Q_1)}\leq C\delta(z')^{3/2}\|\mathcal{W}\|_{W^{2,4}(Q_1)}+C\|\mathcal{W}\|_{W^{1,4}(Q_1)}.
\end{align*}

Now applying Theorem \ref{thmWmq} with $m=0$ and $q=4$,
\begin{align*}
	&\|\mathcal{W}\|_{W^{2,4}(Q_{1/2})}+\|\mathcal{G}-\mathcal{G}_{1}\|_{W^{1,4}(Q_{1/2})}\nonumber\\
	\leq&\, C\left(\|\mathcal{F}-\mathcal{M}\|_{L^{4}(Q_1)}+\|\mathcal{W}\|_{W^{1,4}(Q_1)}+\|\mathcal{G}-\mathcal{G}_{1}\|_{L^{4}(Q_1)}\right)\nonumber\\
	\leq&\, C\left(\|\mathcal{F}\|_{L^{4}(Q_1)}+\delta(z')^{3/2}\|\mathcal{W}\|_{W^{2,4}(Q_1)}+\|\mathcal{W}\|_{W^{1,4}(Q_1)}+\|\mathcal{G}-\mathcal{G}_{1}\|_{L^{4}(Q_1)}\right).
	\end{align*}
Together with the interpolation inequality
\begin{equation*}
\|\mathcal{W}\|_{W^{1,4}(Q_{1})}\leq \theta_0\|\mathcal{W}\|_{W^{2,4}(Q_{1})} +C(\theta_{0})\|\mathcal{W}\|_{L^{4}(Q_{1})}, \quad0<\theta_0<1,
\end{equation*}
we have
\begin{align*}
	\|\mathcal{W}\|_{W^{2,4}(Q_{1/2})}+&\|\mathcal{G}-\mathcal{G}_{1}\|_{W^{1,4}(Q_{1/2})}\leq\,C\left(\|\mathcal{F}\|_{L^{4}(Q_1)}+\|\mathcal{G}-\mathcal{G}_{1}\|_{L^{4}(Q_1)}\right)\nonumber\\
&\,\quad\quad\quad +(C\delta(z')^{3/2}+C\theta_{0})\|\mathcal{W}\|_{W^{2,4}(Q_1)}+C(\theta_{0})\|\mathcal{W}\|_{L^{4}(Q_1)}.
	\end{align*}
Choosing $\theta_0$ small enough such that $C\theta_0\leq\frac{1}{4}$, and after shrinking $\epsilon,R$ if necessary, such that $C\delta(z')^{3/2}\leq\frac{1}{4}$, then 
\begin{align*}
	\|\mathcal{W}\|_{W^{2,4}(Q_{1/2})}+\|\mathcal{G}-\mathcal{G}_{1}\|_{W^{1,4}(Q_{1/2})}
\leq&\,\frac{1}{2}\|\mathcal{W}\|_{W^{2,4}(Q_1)}+C\left(\|\mathcal{W}\|_{L^{4}(Q_1)}+\|\mathcal{F}\|_{L^{4}(Q_1)}+\|\mathcal{G}-\mathcal{G}_{1}\|_{L^{4}(Q_1)}\right).
	\end{align*}
By replacing the cutoff function, we can control the left hand side estimates on $Q_{s}$ with those on $Q_{t}$ for any $0<s<t<1$. Making use of Lemma \ref{itLemma}, together with $\|\mathcal{W}\|_{L^{4}(Q_{1})}\leq\,C\|\nabla\mathcal{W}\|_{L^{2}(Q_1)}$, we get
\begin{align}\label{wmpfz1}
\|\mathcal{W}\|_{W^{2,4}(Q_{1/2})}+\|\mathcal{G}-\mathcal{G}_{1}\|_{W^{1,4}(Q_{1/2})}
\leq&\, C\left(\|\nabla\mathcal{W}\|_{L^{2}(Q_1)}+\|\mathcal{F}\|_{L^{4}(Q_1)}+\|\mathcal{G}-\mathcal{G}_{1}\|_{L^{4}(Q_1)}\right).
\end{align}

In order to control the last term involving $\mathcal{G}$, we use the interpolation inequality,  
	\begin{equation}\label{sz2}
		\|\mathcal{G}-\mathcal{G}_{1}\|_{L^{4}(Q_1)}\leq \theta_{1} \|\mathcal{G}-\mathcal{G}_{1}\|_{L^{\infty}(Q_1)}+\frac{C}{\theta_{1}}\|\mathcal{G}-\mathcal{G}_{1}\|_{L^{2}(Q_1)},
\end{equation}	
for some small $\theta_{1}$ to be determined later. By virtue of \cite[Lemma 3.2]{LX}, we have
	\begin{equation}\label{sz3}
	\|\mathcal{G}-\mathcal{G}_{1}\|_{L^{2}(Q_1)}\leq C\|\mathcal{F}-\mathcal{M}\|_{L^{2}(Q_1)}+C\|\nabla\mathcal{W}\|_{L^{2}(Q_1)}.
	\end{equation}
By the same argument as before, we have
\begin{align*}
	\|\delta(z'){\widetilde{\mathcal{V}}}\cdot\nabla\mathcal{W}\|_{L^2(Q_1)}+	\|\delta(z')\mathcal{W}\cdot\nabla\widetilde{\mathcal{V}}\|_{L^2(Q_1)}\leq C\|\nabla\mathcal{W}\|_{L^2(Q_1)},
\end{align*}
and
\begin{align*}
	\|\delta(z')\mathcal{W}\cdot\nabla\mathcal{W}\|_{L^2(Q_1)}\leq\, C\delta(z') \|\nabla\mathcal{W}\|_{L^2(Q_1)}\|\mathcal{W}\|_{L^{\infty}(Q_1)}\leq  C\delta(z')^{3/2}\|\mathcal{W}\|_{W^{1,4}(Q_1)}.
\end{align*} 
Using the interpolation inequality, 
\begin{equation*}
	\|\mathcal{W}\|_{W^{1,4}(Q_{1})}\leq \theta_1\|\mathcal{W}\|_{W^{2,4}(Q_{1})} +\frac{C}{\theta_{1}}\|\mathcal{W}\|_{L^{4}(Q_{1})},
\end{equation*}
where $\theta_1$ is chosen to be the same constant as in \eqref{sz2}, together with $\|\mathcal{W}\|_{L^{4}(Q_{1})}\leq\,C\|\nabla\mathcal{W}\|_{L^{2}(Q_1)}$, we have
\begin{align*}
	\|\mathcal{F}-\mathcal{M}\|_{L^{2}(Q_1)}\leq C\theta_1\delta(z')^{3/2}\|\mathcal{W}\|_{W^{2,4}(Q_1)}+C(\theta_1)\|\nabla\mathcal{W}\|_{L^{2}(Q_1)}+C\|\mathcal{F}\|_{L^{2}(Q_1)}.
\end{align*}
Substituting this into \eqref{sz3}, and combining with \eqref{sz2} leads to
\begin{align*}
	\|\mathcal{G}-\mathcal{G}_{1}\|_{L^{4}(Q_1)}&\leq\theta_{1} \|\mathcal{G}-\mathcal{G}_{1}\|_{L^{\infty}(Q_1)}+C\delta(z')^{3/2}\|\mathcal{W}\|_{W^{2,4}(Q_1)}+C(\theta_1)\Big(\|\mathcal{F}\|_{L^{2}(Q_1)}+\|\nabla\mathcal{W}\|_{L^{2}(Q_1)}\Big).
\end{align*}
In view of \eqref{wmpfz1},
\begin{align*}
&\|\mathcal{W}\|_{W^{2,4}(Q_{1/2})}+\|\mathcal{G}-\mathcal{G}_{1}\|_{W^{1,4}(Q_{1/2})}\leq \,C\theta_1\|\mathcal{G}-\mathcal{G}_{1}\|_{L^{\infty}(Q_1)}\nonumber\\
&\quad\quad\quad\quad+C\delta(z')^{3/2}\|\mathcal{W}\|_{W^{2,4}(Q_1)}+C(\theta_1)\left(\|\nabla\mathcal{W}\|_{L^{2}(Q_1)}+\|\mathcal{F}\|_{L^{4}(Q_1)}\right).
\end{align*}
After shrinking $\epsilon,R$ if necessary, such that $C\delta(z')^{3/2}\leq \frac{1}{2}$, it follows from Lemma \ref{itLemma} that
\begin{align*}
	&\|\mathcal{W}\|_{W^{2,4}(Q_{1/2})}+\|\mathcal{G}-\mathcal{G}_{1}\|_{W^{1,4}(Q_{1/2})}
	\leq\, C\theta_1\|\mathcal{G}-\mathcal{G}_{1}\|_{L^{\infty}(Q_1)}+C(\theta_1)\left(\|\nabla\mathcal{W}\|_{L^{2}(Q_1)}+\|\mathcal{F}\|_{L^{4}(Q_1)}\right).
\end{align*}
By the Sobolev embedding theorem, 
	\begin{align*}
		\|\nabla\mathcal{W}\|_{L^{\infty}(Q_{1/2})}+&\|\mathcal{G}-\mathcal{G}_{1}\|_{L^{\infty}(Q_{1/2})}\leq\, C\big(\|\mathcal{W}\|_{W^{2,4}(Q_{1/2})}+\|\mathcal{G}-\mathcal{G}_{1}\|_{W^{1,4}(Q_{1/2})}\big).
	\end{align*}
Now choosing $\theta_{1}$ such that $C\theta_1=\frac{1}{2}$, and making use of Lemma \ref{itLemma} again yields  
\begin{align*}
		\|\nabla\mathcal{W}\|_{L^{\infty}(Q_{1/2})}+\|\mathcal{G}-\mathcal{G}_{1}\|_{L^{\infty}(Q_{1/2})}\leq C\left(	\|\nabla\mathcal{W}\|_{L^{2}(Q_1)}+\|\mathcal{F}\|_{L^{\infty}(Q_1)}\right).
	\end{align*}
Rescaling  to $({\bf w},q)$, there holds 
	\begin{align*}
		&\|\nabla {\bf w}\|_{L^{\infty}(\Omega_{\delta(z')/2}(z))}+\|q-q_{\delta(z');z'}\|_{L^{\infty}(\Omega_{\delta(z')/2}(z))}\nonumber\\
		\leq&\,
		\frac{C}{\delta(z')}\left(\delta(z')^{1-{d}/{2}}\|\nabla {\bf w}\|_{L^{2}(\Omega_{\delta(z')}(z))}+\delta(z')^2\|{\bf f}\|_{L^{\infty}(\Omega_{\delta(z')}(z))} \right)\nonumber\\
		\leq&\,
		C\left(\delta(z')^{-{d}/{2}}\|\nabla {\bf w}\|_{L^{2}(\Omega_{\delta(z')}(z))}+\delta(z')\|{\bf f}\|_{L^{\infty}(\Omega_{\delta(z')}(z))} \right).
	\end{align*}
Thus, \eqref{nsW2pstokes} holds. 

Applying Theorem \ref{thmWmq} with $m=1$ and $q=4$, 
	\begin{align*}
		&\|\mathcal{W}\|_{W^{3,4}(Q_{1/2})}+\|\mathcal{G}-\mathcal{G}_{1}\|_{W^{2,4}(Q_{1/2})}\nonumber\\
		\leq&\, C\left(\|\mathcal{F}\|_{W^{1,4}(Q_{2/3})}+\|\mathcal{M}\|_{W^{1,4}(Q_{2/3})}+\|\mathcal{W}\|_{W^{2,4}(Q_{2/3})}+\|\mathcal{G}-\mathcal{G}_{1}\|_{W^{1,4}(Q_{2/3})}\right).
	\end{align*}
By repeating the process above, we conclude that \eqref{Wmpstokes} holds. The proof of Proposition \ref{nsdpro} is finished.
\end{proof}

\section{Proofs for the Main Estimates in 3D}\label{sec4}
This section is dedicated to proving the main estimates in Proposition \ref{propu133D}, \ref{propuhe111} and \ref{lemCialpha3D}, by utilizing the framework established in Section \ref{sec3}. 

\subsection{Proof of Proposition \ref{propu133D} for $({\bf u}_1^3,p_1^3)$: the Nonlinear Part }
By some direct calculations, we have
\begin{align*}
	\partial_{x_j}k(x)=-\frac{2x_{j}}{\delta(x')}k(x),\quad\,j=1,2,\quad
	\partial_{x_3}k(x)
	=\frac{1}{\delta(x')},\quad\hbox{in}\ \Omega_{2R}.
\end{align*}
Further, recalling the definition of ${\bf v}_{1}^{3}$, \eqref{v133D}, in $\Omega_{2R}$,
\begin{align}
	\partial_{x_{j}}({\bf v}_{1}^{3})^{(j)}&=\left(\frac{3}{\delta(x')}-\frac{18x_{j}^{2}}{\delta(x')^{2}}\right)k(x)^{2}-\frac{1}{4}\left(\frac{3}{\delta(x')}-\frac{6x_{j}^{2}}{\delta(x')^{2}}\right),\quad\,j=1,2;\label{estv131}\\
	\partial_{x_{l}}({\bf v}_{1}^{3})^{(j)}&=\frac{18x_{1}x_{2}}{\delta(x')^{2}}k(x)^{2}+\frac{3}{2}\frac{x_{1}x_{2}}{\delta(x')^{2}},~ \partial_{x_3}({\bf v}_{1}^{3})^{(j)}=\frac{6x_{j}}{\delta(x')^2}k(x),~j,l=1,2,j\neq\,l;\nonumber
\end{align}
while, for $j=3$,
\begin{equation}\label{estv133}
	|\partial_{x_l}({\bf v}_{1}^{3})^{(3)}|\leq\frac{C|x_l|}{\delta(x')},\quad l=1,2,
\end{equation}
and
\begin{equation*}\label{estv1333}
	\partial_{x_3}({\bf v}_{1}^{3})^{(3)}=\left(\frac{18|x'|^{2}}{\delta(x')^{2}}-\frac{6}{\delta(x')}\right)k(x)^{2}-\frac{3}{2}\left(\frac{|x'|^{2}}{\delta(x')^{2}}-\frac{1}{\delta(x')}\right).
\end{equation*}
Hence,
\begin{equation}\label{estv133311}
	|\nabla{\bf v}_{1}^{3}|\leq\,C\left(\frac{|x'|}{\delta(x')^{2}}+\frac{1}{\delta(x')}\right).
\end{equation}

On the other hand, by virtue of the construction of ${\bar p}_1^3$ in \eqref{p133D}, we have
\begin{equation*}
	\mu\partial_{x_3}({\bf v}_{1}^{3})^{(3)}-\overline{p}_1^3=\frac{3}{2}\frac{\mu}{\delta(x')^2}-\frac{3}{2}\mu\left(\frac{|x'|^{2}}{\delta(x')^{2}}-\frac{1}{\delta(x')}\right),
\end{equation*}
and further differentiating it with respect to $x_{3}$, 
\begin{equation}\label{estv133p13}
	\mu\partial_{x_3x_3} ({\bf v}_{1}^3)^{(3)}-\partial_{x_3}\overline{p}_1^3=0.
\end{equation}
Here we would like to emphasize that this is the main reason for constructing this type of $\overline{p}_1^3$. While,
\begin{equation*}
	\partial_{x_j}\overline{p}_1^3=\frac{6\mu x_j}{\delta(x')^{3}}+\frac{72x_j\mu }{\delta(x')^2}\left(1-\frac{2|x'|^{2}}{\delta(x')}\right)k(x)^{2}, \quad j=1,2,
\end{equation*}
and 
$$\partial_{x_3x_{3}}({\bf v}_{1}^{3})^{(j)}=\frac{6x_{j}}{\delta(x')^{3}}.$$
So that
\begin{equation}\label{estv131p13}
	\Big|\mu\partial_{x_3x_3} ({\bf v}_{1}^3)^{(j)}-\partial_{x_j}\overline{p}_1^3\Big|=\left|-\frac{72x_j\mu}{\delta(x')^2}\left(1-\frac{2|x'|^{2}}{\delta(x')}\right)k(x)^{2}\right|\leq\frac{C|x'|}{\delta(x')^2}.
\end{equation}
For other second order derivatives, we have 
\begin{equation*}
	|\partial_{x_1x_1}({\bf v}^{3}_{1})^{(j)}|,|\partial_{x_2x_2}({\bf v}^{3}_{1})^{(j)}|\leq\frac{C|x'|}{\delta(x')^2},\quad |\partial_{x_1x_1}({\bf v}^{3}_{1})^{(3)}|,|\partial_{x_2x_2}({\bf v}^{3}_{1})^{(3)}|\leq\frac{C}{\delta(x')}.
\end{equation*}

\begin{lemma}Let $({\bf w}_{3},q_{3})$ be the solution to \eqref{w35}. Then
\begin{equation}\label{ztnl3d}
\int_{\Omega}|\nabla {\bf w}_1^3|^2\leq C
\end{equation}
\end{lemma}

\begin{proof}
Recalling \eqref{f333},
$${\bf f}=\mu\,\Delta {\bf v}_1^3-\nabla \overline{p}_1^3-C_1^3\,\widetilde{\bf v}\cdot \nabla {\bf v}_1^3-\widetilde{\bf v}\cdot \nabla {\bf u}^{\#3},\quad\mbox{and}\quad~
\widetilde{\bf v}:=C_1^3{\bf v}_1^3+ {\bf u}^{\#3}.$$
We estimate the components in ${\bf f}$. First, it follows from \eqref{estv133p13}--\eqref{estv131p13} that 
\begin{equation}\label{estf13}
	\left|\mu\Delta{\bf v}_{1}^3-\nabla\overline{p}_1^3\right|\leq\frac{C|x'|}{\delta(x')^{2}},\quad \mbox{in}~\Omega_{2R}.
\end{equation}
For other terms, set
$${\bf u}^{\#3}:={\bf v}^{\#3}+{\bf w}^{\#3},$$
where
$${\bf v}^{\#3}=\sum_{(i,\alpha)\neq (1,3)}{\bf v}_i^\alpha,\quad\mbox{and}~~ {\bf w}^{\#3}=\sum_{(i,\alpha)\neq (1,3)}({\bf u}_i^\alpha-{\bf v}_i^\alpha)+{\bf u}_0.$$

Note that from \eqref{v133D},
\begin{align*}
|{\bf v}_1^3\cdot\nabla{\bf v}_1^3|=&\sum_{j=1,2}\Big(|({\bf v}_1^3)^{(3)}\partial_{x_{3}}({\bf v}_1^3)^{j}|+\sum_{i=1,2}|({\bf v}_1^3)^{(i)}\cdot\nabla_{x_{i}}({\bf v}_1^3)^{(j)}|\Big)+|{\bf v}_1^3\cdot\nabla({\bf v}_1^3)^{(3)}|\\
\leq&\,C\Big(\frac{|x'|}{\delta(x')^2}+\frac{1}{\delta(x')}\Big).
\end{align*}
Similarly,
$$|{\bf v}_2^3\cdot\nabla{\bf v}_1^3(x)|\leq C\Big(\frac{|x'|}{\delta(x')^2}+\frac{1}{\delta(x')}\Big).$$
By virtue of \eqref{zydfzgs11} in Remark \ref{rem28}, 
$$|{\bf u}^{\#3}-C_2^3{\bf u}_2^3|\leq|{\bf w}^{\#3}-C_2^3{\bf w}_2^3|+|{\bf v}^{\#3}-C_2^3{\bf v}_2^3|\leq\,C.$$
Then, combining with \eqref{boundofv13} and \eqref{estv133311} yields 
\begin{align*}
	|{\bf u}^{\#3}\cdot\nabla{\bf v}_1^3|\leq \Big( |{\bf u}^{\#3}-C_2^3{\bf u}_2^3|+|C_2^3{\bf w}_2^3|\Big)|\nabla{\bf v}_1^3|+|C_2^3{\bf v}_2^3\cdot\nabla{\bf v}_1^3|\leq C\Big(\frac{|x'|}{\delta(x')^2}+\frac{1}{\delta(x')}\Big).
\end{align*}
Hence
\begin{align}\label{fdgj3D}
	|C_1^3\,\widetilde{\bf v}\cdot \nabla {\bf v}_1^3(x)|&\leq C|{\bf v}_1^3\cdot\nabla{\bf v}_1^3|+C|{\bf u}^{\#3}\cdot\nabla{\bf v}_1^3|\leq C\Big(\frac{|x'|}{\delta(x')^2}+\frac{1}{\delta(x')}\Big).
\end{align}
Similarly, 
\begin{equation}\label{fdgj3D2}
	|\widetilde{\bf v}\cdot \nabla {\bf u}^{\#3}(x)|\leq C|{\bf v}_1^3\cdot \nabla {\bf u}^{\#3}|+C|{\bf u}^{\#3}\cdot \nabla {\bf u}^{\#3}| \leq C\Big(\frac{|x'|}{\delta(x')^2}+\frac{1}{\delta(x')}\Big).
\end{equation}
By means of \eqref{estf13}, \eqref{fdgj3D} and \eqref{fdgj3D2}, we deduce
\begin{align}\label{fz9}
	|{\bf f}(x)|\leq\,C\Big(\frac{|x'|}{\delta(x')^2}+\frac{1}{\delta(x')}\Big).
\end{align}

By using the H{\"o}lder's inequality and the Hardy inequality, it follows from \eqref{fz9} that
\begin{align}\label{int-fw?}
	\Big| \int_{\Omega_{R}}\sum_{j=1}^{3}{\bf f}^{(j)}({\bf w}_1^3)^{(j)}\Big|\leq\|\delta(x'){\bf f}\|_{L^2(\Omega_{R})}\Big\|\frac{{\bf w}_1^3}{x_3}\Big\|_{L^2(\Omega_{R})}\leq 
	C\left(\int_{\Omega}|\nabla {\bf w}_1^3|^2\mathrm{d}x\right)^{1/2}.
\end{align}
Moreover, by using \eqref{zydfzgs11}, we have, in $\Omega_{R}$,
\begin{equation*}
|{\bf u}^{\#3}(x)|\leq \sum_{(i,\alpha)\neq (1,3)}|{\bf u}_i^\alpha-{\bf v}_i^\alpha|+|{\bf u}_0|+\sum_{(i,\alpha)\neq (1,3)}|{\bf v}_i^\alpha|\leq \frac{C}{\sqrt{\delta(x')}},
\end{equation*}
and
\begin{equation*}
	|\nabla{\bf u}^{\#3}(x)|\leq \sum_{(i,\alpha)\neq (1,3)}|\nabla({\bf u}_i^\alpha-{\bf v}_i^\alpha)|+|\nabla{\bf u}_0|+\sum_{(i,\alpha)\neq (1,3)}|\nabla{\bf v}_i^\alpha|\leq \frac{C}{\delta(z')^{3/2}}.
\end{equation*}
Hence, by using Proposition \ref{lemma11} and \eqref{v133D}, we obtain, in $\Omega_{R}$,
\begin{align}\label{js1}
	|\widetilde{\bf v}(x)|=|( C_1^3{\bf v}_1^3+ {\bf u}^{\#3})|\leq C|{\bf v}_1^3|+ |{\bf u}^{\#3}|\leq \frac{C}{\sqrt{\delta(x')}},
\end{align}
and
\begin{align}\label{js2}
	|\nabla\widetilde{\bf v}(x)|=|( C_1^3\nabla{\bf v}_1^3+ \nabla{\bf u}^{\#3})|\leq C|\nabla{\bf v}_1^3|+ |\nabla{\bf u}^{\#3}|\leq  \frac{C}{\delta(x')^{3/2}}.
\end{align}
Hence, by making use of \eqref{int-fw?}--\eqref{js2} and Lemma \ref{lemmaenergy}, we deduce
\eqref{ztnl3d} holds.
\end{proof}

Using \eqref{fz9}, we have
	\begin{align}\label{L2_f13???}
		\int_{\Omega_{s}(z')}|{\bf f}_{1}^{3}(x)|^{2}\mathrm{d}x\leq\frac{Cs^2}{\delta(z')^{3}}(s^2+|z'|^2+\delta(z')^2).
\end{align}
By virtue of \eqref{L2_f13???}, we derive the following local energy estimate.
\begin{lemma}\label{nszyly333}
	Let $({\bf w}, q)$ be the solution of \eqref{w35}. Then 
	\begin{equation}\label{nssfz}
		\int_{\Omega_\delta(z')} |\nabla {\bf w}|^2\leq C \delta(z')^2.
	\end{equation}
\end{lemma}
\begin{proof}
Here we adopt the iteration scheme, which is in spirit similar to that used in \cite[Lemma 3.4]{LX2}.
	Denote $F(t):=\int_{\Omega_{t}(z_1)}|\nabla {\bf w}|^{2}$. Substituting \eqref{L2_f13???} into \eqref{nsiterating1}, we have
		\begin{align}\label{nsiteration3D}
			F(t)\leq &\Big(\frac{1}{4}+(\frac{c_{0}\delta(z')}{s-t})^{2}+\frac{1}{8}\frac{\delta(z')}{s-t}\Big) F(s)+C\Big(\!(s-t)^{2}\!\!+\!\delta(z')^2\Big)\frac{s^2(s^2\!\!+\!|z'|^2\!+\!\delta(z')^2)}{\delta(z')^3},
		\end{align}
where constant $c_{0}>1$ now is fixed.  Let $k_{0}=\left[\frac{1}{4c_{0}\sqrt{\delta(z')}}\right]+1$ and $t_{i}=\delta(z')+2c_{0}i\delta(z'), i=0,1,2,\dots,k_{0}$. Then we take $s=t_{i+1}$ and $t=t_{i}$ in \eqref{nsiteration3D} to derive an iteration formula
		\begin{equation*}
			F(t_{i})\leq \frac{9}{16}F(t_{i+1})+C(i+1)^2\delta(z')^2.
		\end{equation*}
		After $k_{0}$ iterations, and using \eqref{ztnl3d}, we obtain, for sufficiently small $\varepsilon$ and $|z'|$,
		\begin{equation*}
			F(t_0)\leq \left(\frac{9}{16}\right)^{k_{0}}F(t_{k_{0}})
			+C\delta(z')^2\sum\limits_{l=0}^{k_{0}-1}\left(\frac{9}{16}\right)^{l}(l+1)^2\leq C\delta(z')^2.
	\end{equation*}
The proof of Lemma \ref{nszyly333} is completed.
\end{proof}

We now are  ready to prove Proposition \ref{propu133D} for the nonlinear part $({\bf u}_1^3,p_1^3)$.
\begin{proof}[Proof of Proposition \ref{propu133D} for $({\bf u}_1^3,p_1^3)$.]
  Taking advantage of Proposition \ref{nsdpro}, \eqref{fz9} and \eqref{nssfz}, we have
	\begin{align*}
		\|\nabla {\bf w}_1^3\|_{L^{\infty}(\Omega_{\delta(z')/2}(z))}
		\leq&\,
		\,C\left(\delta(z')^{-3/2}\|\nabla {\bf w}_1^3\|_{L^{2}(\Omega_{\delta(z')}(z))}+\delta(z')\|{\bf f}_1^3\|_{L^{\infty}(\Omega_{\delta(z')}(z))} \right)\nonumber\\
		\leq&\,\,C\left(\delta(z')^{-1/2}+\frac{|z'|}{\delta(z')}\right)
		\leq\,\frac{C}{\sqrt{\delta(z')}},
	\end{align*}
	and 
	\begin{align*}
		&\|\nabla^2 {\bf w}_1^3\|_{L^{\infty}(\Omega_{\delta(z')/2}(z))}+\|\nabla q_1^3\|_{L^{\infty}(\Omega_{\delta(z')/2}(z))}\nonumber\\
		\leq&\,
		\,C\Big(\delta(z')^{-5/2}\|\nabla {\bf w}_1^3\|_{L^{2}(\Omega_{\delta(z')}(z))}+\|{\bf f}_1^3\|_{L^{\infty}(\Omega_{\delta(z')}(z))}+\delta(z')\|\nabla{\bf f}_1^3\|_{L^{\infty}(\Omega_{\delta(z')}(z))}\Big)\nonumber\\
		\leq&\,\,\frac{C}{\delta(z')^{3/2}},
	\end{align*}
	Recalling  ${\bf w}_1^3={\bf u}_1^3-{\bf v}_1^3$, $q_1^3=p_1^3-\bar p_1^3$, and using \eqref{estv131}--\eqref{estv133}, we complete the proof of Proposition \ref{propu133D}. 
\end{proof}

\subsection{Proof of Proposition \ref{propuhe111}}
By using the same argument as in Proposition \ref{propu133D}, we prove Proposition \ref{propuhe111}.

\begin{proof}[Proof of Proposition \ref{propuhe111}]
	Given a point $z=(z', z_{3})\in \Omega_{R}$,
	denote ${\bf u}_{\psi}^{3}={\bf u}_1^3+{\bf u}_2^3-\boldsymbol{\psi}_3$ and $p_3=p_1^3+p_2^3$, then $\nabla{\bf u}_{\psi}^{3}=\nabla({\bf u}_1^3+{\bf u}_2^3)$. Moreover, by virtue of \eqref{equ_v1} and \eqref{equ_v13}, we can derive $({\bf u}_{\psi}^{3}, p_3)$ verify the following boundary value problem,
	\begin{equation}\label{equ_v12Dd11?}
		\begin{cases}
			\nabla\cdot\sigma[{\bf u}_{\psi}^{3},p_3]={\bf u}\cdot\nabla{\bf u},\quad\nabla\cdot {\bf u}_{\psi}^{3}=0,&\mathrm{in}~\Omega,\\
			{\bf u}_{\psi}^{3}=0,&\mathrm{on}~\partial{D}_{1}\cup\partial{D_{2}},\\
			{\bf u}_{\psi}^{3}=-\boldsymbol{\psi}_3,&\mathrm{on}~\partial{D}.
		\end{cases}
	\end{equation}
	Since ${\bf u}_{\psi}^{3}=0$ on $\partial D_1\cup\partial D_2$, by using the Poincar\'e inequality, we have 
	\begin{align*}\label{estwDwl21}
		\int_{\Omega_{s}(z')}|{{\bf u}_{\psi}^{3}}|^2 \leq
		C\delta(z')^2\int_{\Omega_{s}(z_1)}|\nabla {{\bf u}_{\psi}^{3}}|^2.
	\end{align*}

Set $${\bf u}_{2}^{\#3}={\bf u}_a+\sum_{\alpha\neq 3}C_{2}^{\alpha}({\bf u}_{1}^{\alpha}+{\bf u}_{2}^{\alpha})+C_2^3\boldsymbol{\psi}_3,$$ 
	where ${\bf u}_a=\sum_{\alpha=1}^{6}|C_{1}^{\alpha}-C_{2}^{\alpha}|{\bf u}_{1}^{\alpha}+{\bf u}_{0}.$
	Then, we obtain
	\begin{equation*}
		{\bf u}\cdot\nabla{\bf u}=({\bf u}_{2}^{\#3}+C_2^3{\bf u}_{\psi}^{3})\cdot\nabla({\bf u}_{2}^{\#3}+C_2^3{\bf u}_{\psi}^{3}).
	\end{equation*}
Then, we can rewrite \eqref{equ_v12Dd11?} as
\begin{equation*}
\left\{\begin{split}
\mu \Delta {\bf u}_{\psi}^{3}&=(C_2^3)^2{\bf u}_{\psi}^{3} \cdot\nabla{\bf u}_{\psi}^{3}+C_2^3{\bf u}_{2}^{\#3} \cdot\nabla{\bf u}_{\psi}^{3}\\
&\quad\quad\quad\quad+C_2^3{\bf u}_{\psi}^{3} \cdot\nabla{\bf u}_{2}^{\#3}+{\bf u}_{2}^{\#3}\cdot\nabla{\bf u}_{2}^{\#3},&\mathrm{in}&~\Omega,\\
			\nabla\cdot {\bf u}_{\psi}^{3}&=0,&\mathrm{in}&~\Omega,\\
		{\bf u}_{\psi}^{3}&=0,~\mathrm{on}~\partial{D}_{1}\cup\partial{D_{2}},\quad
			{\bf u}_{\psi}^{3}=-\boldsymbol{\psi}_3,\mathrm{on}~\partial{D}.
\end{split}\right.
\end{equation*}
 By virtue of Propositions \ref{propu113D1}, \ref{propu4563D}--\ref{prop1.7}, \ref{propu133D} and \eqref{zydfzgs11}, we can derive
    	\begin{align}\label{dr1}
    		|C_2^3{\bf u}_{2}^{\#3}|\leq C\sum_{\alpha=1}^6 |{\bf u}_1^\alpha-{\bf v}_1^\alpha|+C\sum_{\alpha\neq 3}|{\bf u}_1^\alpha+{\bf u}_2^\alpha|+C\sum_{\alpha=1}^6|{\bf v}_1^\alpha|+C\leq \frac{C}{\sqrt{\delta(x')}}.
    	\end{align}
    	Similarly, one can verify
    	\begin{align}\label{dr2}
    		|C_2^3\nabla{\bf u}_{2}^{\#3}(x)|\leq \frac{C}{\delta(x')^{3/2}}.
    	\end{align}
    	Next, we pay our attention to the non-homogeneous terms ${\bf u}_{2}^{\#3}\cdot\nabla{\bf u}_{2}^{\#3}$.
	By using Propositions \ref{propu113D1}, \ref{propu4563D}, \ref{propu133D} and mean-value theorem, since ${\bf w}_1^\alpha=0$ on $\partial D_1$, a direct calculation yields, $\alpha\neq 3$, $x\in\Omega_{R}$,
	\begin{align}\label{wsc1}
		|{\bf w}_1^\alpha(x)|&=	|{\bf w}_1^\alpha(x)-{\bf w}_1^\alpha(x',\varepsilon/2+|x'|^2/2)|\leq C \delta(x'),
	\end{align}
and
	\begin{align}\label{wsc2}
	|{\bf w}_1^3(x)|&=	|{\bf w}_1^3(x)-{\bf w}_1^3(x',\varepsilon/2+|x'|^2/2)|\leq C \sqrt{\delta(x')}.
\end{align}
Thus, by using \eqref{wsc1}, we deduce
	\begin{align*}
\Big|\sum_{\alpha\neq 3}|C_{1}^{\alpha}-C_{2}^{\alpha}|{\bf u}_{1}^{\alpha}\cdot\nabla {\bf u}_1^3\Big|
\leq&\, C\Big|\sum_{\alpha\neq 3}{\bf v}_{1}^{\alpha}\cdot\nabla {\bf v}_1^3\Big|+C\Big|\sum_{\alpha\neq 3}{\bf w}_{1}^{\alpha}\cdot\nabla {\bf v}_1^3\Big|+C\Big|\sum_{\alpha\neq 3}{\bf v}_{1}^{\alpha}\cdot\nabla {\bf w}_1^3\Big|\\
		\leq&\, C\Big|\sum_{\alpha\neq 3}{\bf v}_{1}^{\alpha}\cdot\nabla {\bf v}_1^3\Big|+\frac{C|x'|}{\delta(x')}\leq \frac{C\sum_{\alpha\neq3,4}|C_1^\alpha-C_2^\alpha|}{\delta(x')}+\frac{C}{\sqrt{\delta(x')}}.
	\end{align*}
	where we use the fact that $|\partial_{x_3}({\bf v}_1^3)^{(j)}|\sim \frac{C|x'|}{\delta(x')^2}$, $j=1,2$ are the biggest term in the gradient matrix of ${\bf v}_1^3$, and $|({\bf v}_1^\alpha)^{(3)}| \leq C|x'|$, $\alpha\neq 3$. Similarly, in view of \eqref{wsc2}, we have
	\begin{align*}
		\Big|\left(C_{1}^{3}-C_{2}^{3}\right){\bf u}_{1}^{3}\cdot\nabla {\bf u}_1^3\Big|\leq \frac{C|C_{1}^{3}-C_{2}^{3}||x'|}{\delta(x')^2}+\frac{C|C_{1}^{3}-C_{2}^{3}|}{\delta(x')}+C.
	\end{align*}
	By using mean-value theorem and Proposition \ref{prop1.7}, we have
	\begin{equation*}
		\Big|{\bf u}_0+\sum_{\alpha\neq 3}C_{2}^{\alpha}({\bf u}_{1}^{\alpha}+{\bf u}_{2}^{\alpha}-\boldsymbol{\psi}_\alpha)\Big|\leq C\delta(x').
	\end{equation*}
 Hence,
	\begin{equation*}
		\Big|\Big({\bf u}_0+\sum_{\alpha\neq 3}C_{2}^{\alpha}({\bf u}_{1}^{\alpha}+{\bf u}_{2}^{\alpha})\Big)\cdot \nabla {\bf u}_1^3\Big|\leq C\Big|\sum_{\alpha\neq 3}\boldsymbol{\psi}_\alpha\cdot \nabla{\bf v}_1^3\Big|+\frac{C|x'|}{\delta(x')}\leq\frac{C}{\delta(x')}.
	\end{equation*}
	The other terms can be estimated similarly. Hence, 
	$$|{\bf u}_{2}^{\#3}\cdot\nabla{\bf u}_{2}^{\#3}| \leq\frac{C\sum_{\alpha\neq4}|C_1^\alpha-C_2^\alpha|}{\delta(x')} +\frac{C|C_1^3-C_2^3||x'|}{\delta(x')^2}+\frac{C|C_1^4-C_2^4||x'|}{\delta(x')}+C.$$
	Then, by using \eqref{dr1}, \eqref{dr2} and the same argument as that in Proposition \ref{propu133D}, we derive
	$$\|\nabla {\bf u}_{\psi}^{3}\|_{L^{\infty}(\Omega_{\delta(z')/2}(x'))}\leq \frac{C|C_1^3-C_2^3|}{\sqrt{\delta(x')}}+C\leq \frac{C}{\sqrt{\delta(x')}},$$
	and
	\begin{align*}
		&\|\nabla^2{\bf u}_{\psi}^{3}\|_{L^{\infty}(\Omega_{\delta(z')/2}(x'))}+	\|\nabla p_3\|_{L^{\infty}(\Omega_{\delta(z')/2}(x'))}\\
		\leq&\, \frac{C\sum_{\alpha\neq4}|C_1^\alpha-C_2^\alpha|}{\delta(x')} +\frac{C|C_1^3-C_2^3||x'|}{\delta(x')^2}+\frac{C|C_1^4-C_2^4||x'|}{\delta(x')}+C.
	\end{align*}
	We thus complete the proof.
\end{proof}

\subsection{Proof of Proposition \ref{lemCialpha3D} on $|C_1^\alpha-C_2^\alpha|$}\label{subsec43}

From \eqref{sto-2} and decomposition \eqref{introC}, we have the following linear system of $C_{i}^{\alpha}$:
\begin{equation}\label{equ-decompositon?}
	\sum_{i=1}^2\sum\limits_{\alpha=1}^{6} C_{i}^{\alpha}
	\int_{\partial D_j}{\boldsymbol\psi}_\beta\cdot\sigma[{\bf u}_{i}^\alpha,p_{i}^{\alpha}]\nu
	+\int_{\partial D_j}{\boldsymbol\psi}_\beta\cdot\sigma[{\bf u}_{0},p_{0}]\nu=0,~~\beta= 1,\dots,6,
\end{equation}
where $j=1,2$. Define
\begin{align*}
	a_{ij}^{\alpha\beta}:=-
	\int_{\partial D_j}{\boldsymbol\psi}_\beta\cdot\sigma[{\bf u}_{i}^\alpha,p_{i}^{\alpha}]\nu,\quad
	b_{j}^{\beta}:=
	\int_{\partial D_j}{\boldsymbol\psi}_\beta\cdot\sigma[{\bf u}_{0},p_{0}]\nu.
\end{align*}
Using the integration by parts, \eqref{equ_v1}, and \eqref{equ_u0}, we have, for $(i,\alpha)\neq (1,3)$,
\begin{align*}
	a_{ij}^{\alpha\beta}=\int_{\Omega} \left(2\mu e({\bf u}_{i}^{\alpha}), e({\bf u}_{j}^{\beta})\right)\mathrm{d}x,\quad
	b_{j}^{\beta}=-\int_{\Omega} \left(2\mu e({\bf u}_{0}),e({\bf u}_{j}^\beta)\right)\mathrm{d}x,
\end{align*}
while,
\begin{align*}
	a_{1j}^{3\beta}=\int_{\Omega} \left(2\mu e({\bf u}_{1}^{3}), e({\bf u}_{j}^{\beta})\right)\mathrm{d}x+\int_{\Omega}{\bf u}\cdot\nabla{\bf u}\cdot{\bf u}_j^\beta\mathrm{d}x.
\end{align*}
Thus, \eqref{equ-decompositon?} can be rewritten as
\begin{align}\label{systemC}
	\begin{cases}
		\sum\limits_{\alpha=1}^{6}C_{1}^{\alpha}a_{11}^{\alpha\beta}
		+\sum\limits_{\alpha=1}^{6}C_{2}^{\alpha}a_{21}^{\alpha\beta}
		-b_{1}^{\beta}=0,&\\\\
		\sum\limits_{\alpha=1}^{6}C_{1}^{\alpha}a_{12}^{\alpha\beta}
		+\sum\limits_{\alpha=1}^{6}C_{2}^{\alpha}a_{22}^{\alpha\beta}
		-b_{2}^{\beta}=0.
	\end{cases}\quad \beta=1,\dots,6.
\end{align}
To prove Proposition \ref{lemCialpha3D}, we first estimate  $a_{11}^{\alpha\beta}$ and $b_1^\beta$, $\alpha,\beta=1,\dots,6$. First, for the diagonal elements $a_{11}^{\alpha\alpha}$, $\alpha=1,\dots,6$.
\begin{lemma}\label{lema113D}
	We have
	\begin{align*}
		\frac{1}{C}|\log\varepsilon|\leq a_{11}^{\alpha\alpha}\leq C|\log\varepsilon|,\quad\alpha=1,2,5,6,~~
		\frac{1}{C\varepsilon}\leq a_{11}^{33}\leq \frac{C}{\varepsilon},
		~\mbox{and}~\frac{1}{C}\leq a_{11}^{44}\leq C.
	\end{align*}
\end{lemma}
\begin{proof}
	By using the same argument as that in \cite[(3.49)]{LX2}, we have
	\begin{align*}
\frac{1}{C\varepsilon}\leq \Big|\int_{\Omega} \left(2\mu e({\bf u}_{1}^{3}), e({\bf u}_{1}^{3})\right)\Big|\leq \frac{C}{\varepsilon}
	\end{align*}
In addition, by using Propositions \ref{propu113D1}--\ref{propu133D}, \eqref{v1alpha}, \eqref{v15}, we have
	\begin{align*}
		\Big|\int_{\Omega}{\bf u}\cdot\nabla{\bf u}\cdot{\bf u}_1^3\Big|\leq C\int_{\Omega_{R}}\frac{1}{\delta(x')^2}\leq C|\log\varepsilon|.
	\end{align*}
	Hence, we have $\frac{1}{C\varepsilon}\leq a_{11}^{33}\leq \frac{C}{\varepsilon}$.

The other terms are the same as that in \cite[Lemma 3.8]{LX2}. We omit the proof here.
\end{proof}

For the other elements in $(a_{11}^{\alpha\beta})_{6\times6}$, and $b_1^\beta$, 
\begin{lemma}\label{lema114563D}
	We have
	\begin{align*}
		&|a_{11}^{\alpha\beta}|, ~|a_{11}^{\beta\alpha}|\leq C,\quad\quad\alpha,\beta=1,\dots,6,~\alpha\neq\beta;\\
		&|a_{11}^{\alpha\beta}+a_{21}^{\alpha\beta}|\leq C,\quad\quad\alpha,\beta=1,\dots,6,~(\alpha,\beta)\neq (3,3);\\
	   &| a_{11}^{33}+a_{21}^{33}|\leq  C|C_1^3-C_2^3||\log\varepsilon|+C;
	\end{align*}
	and
	\begin{equation*}
		|b_1^\beta|\leq C,\quad\beta=1,\dots,6.
	\end{equation*}
\end{lemma}

\begin{proof}
Here we mainly deal with the additional integral term in $a_{11}^{3\beta}$. By making use of Propositions \ref{propu113D1}--\ref{propu133D}, \eqref{v1alpha}, \eqref{v15}, we derive, for $\beta\neq 3$,
 \begin{align*}
 \Big|\int_{\Omega}{\bf u}\cdot\nabla{\bf u}\cdot{\bf u}_1^\beta\Big|\leq C\int_{\Omega_{R}}\frac{|x'|}{\delta(x')^2}\leq C.
 \end{align*}
 In view of \cite[Lemma 3.9]{LX2}, we have
 \begin{align*}
 	|b_1^\beta|,~|a_{11}^{\alpha\beta}|, ~|a_{11}^{\beta\alpha}|\leq C,~\alpha,\beta=1,\dots,6,~\alpha\neq\beta.
 \end{align*}
 
 For $a_{11}^{33}+a_{21}^{33}$, by using Propositions \ref{propu113D1}--\ref{propu133D}, we have
 \begin{align*}
| a_{11}^{33}+a_{21}^{33}|=&\int_{\Omega_R}\left(2\mu e({\bf u}_{1}^{3}+{\bf u}_2^3), e({\bf u}_{1}^{3})\right)+\int_{\Omega_R}{\bf u}\cdot \nabla{\bf u}\cdot {\bf u}_{1}^{3}+C
\leq\, C|C_1^3-C_2^3||\log\varepsilon|+C,
 \end{align*}
where we use the fact that
\begin{align*}
\bigg|\int_{\Omega_R}{\bf u}\cdot \nabla{\bf u}\cdot {\bf u}_{1}^{3}\bigg|\leq C|C_1^3-C_2^3|\bigg|\int_{\Omega_R}{\bf v}_1^3\cdot \nabla{\bf v}_1^3\cdot {\bf v}_{1}^{3}\bigg|+C\leq  C|C_1^3-C_2^3||\log\varepsilon|+C.
\end{align*}
The remaining terms can be estimated as before, so we omit the details.
 \end{proof}

\begin{proof}[Proof of Proposition \ref{lemCialpha3D}.] To solve $|C_1^\alpha-C_2^\alpha|$, we use the first six equations in \eqref{systemC}:
	\begin{align*}
		a_{11}(C_{1}-C_{2})=f:=b_1-(a_{11}+a_{21})C_{2},
	\end{align*}
	where $a_{11}:=(a_{11}^{\alpha\beta})_{\alpha,\beta=1}^6$, $C_{i}:=(C_{i}^{1}, \dots, C_i^6)^{\mathrm{T}}$, and  $b_1:=(b_1^1,\dots, b_1^6)^{\mathrm{T}}$. From \cite{BLL2}, the matrix $a_{11}$  is positive definite. By virtue of Lemma \ref{lema114563D}, we have
	\begin{align}\label{fdgj1}
	 |f^\beta|\leq C, i\neq 3,~~|f^3|\leq C|C_1^3-C_2^3||\log\varepsilon|+C.
	\end{align}
	Denote $\mathrm{cof}(A)_{\alpha\beta}$ by the cofactor of $A:=(a_{11}^{\alpha\beta})_{6\times6}$ for simplicity of notation. By  Lemma \ref{lema113D} and Lemma \ref{lema114563D}, we have 
	\begin{equation*}
		\frac{1}{C\varepsilon}|\log\varepsilon|^4\leq\det{A}\leq\frac{C}{\varepsilon}|\log\varepsilon|^4,
	\end{equation*}
and
	\begin{align*}
		\frac{1}{C\varepsilon}|\log\varepsilon|^3\leq&\,\mathrm{cof}(A)_{\alpha\alpha}\leq \frac{C}{\varepsilon}|\log\varepsilon|^3,\quad\alpha=1,2,5,6,\\
		\frac{1}{C}|\log\varepsilon|^4\leq&\,\mathrm{cof}(A)_{33}\leq C|\log\varepsilon|^4,\quad\frac{1}{C\varepsilon}|\log\varepsilon|^4\leq\mathrm{cof}(A)_{44}\leq\frac{C}{\varepsilon}|\log\varepsilon|^4;
\end{align*}
\begin{align*}		
		|\mathrm{cof}(A)_{\alpha\beta}|&\leq \frac{C}{\varepsilon}|\log\varepsilon|^2,\quad\quad\alpha,\beta=1,2,5,6,~\alpha\neq\beta,\\
		|\mathrm{cof}(A)_{\alpha 3}|, |\mathrm{cof}(A)_{3\alpha}|&\leq C|\log\varepsilon|^3,\quad\quad\alpha=1,2,5,6,\\
		|\mathrm{cof}(A)_{43}|, |\mathrm{cof}(A)_{34}|&\leq C|\log\varepsilon|^4,\\
		|\mathrm{cof}(A)_{\alpha 4}|, |\mathrm{cof}(A)_{4\alpha}|&\leq \frac{C}{\varepsilon}|\log\varepsilon|^3,\quad\quad\alpha=1,2,5,6.
	\end{align*}
	Thus, by the Cramer's rule, we obtain
\begin{equation*}
		|C_1^3-C_2^3|\leq\frac{C}{\det{A}}\Big(\sum_{\alpha\neq 3}\mathrm{cof}(A)_{\alpha3}+|f^3|\mathrm{cof}(A)_{33}\Big)\leq C\varepsilon+C\varepsilon|C_1^3-C_2^3|.
\end{equation*}
For sufficiently small $\varepsilon$ such that $C\varepsilon\leq \frac{1}{2}$, we have
\begin{equation*}
	|C_1^3-C_2^3|\leq C\varepsilon.
\end{equation*}
Hence, by using \eqref{fdgj1}, 
\begin{equation*}
	|f^\beta|\leq C, ~i=1,\dots,6.
\end{equation*}
Then
	\begin{equation*}
		|C_1^\alpha-C_2^\alpha|\leq\frac{C}{|\log\varepsilon|},~\alpha=1,2,5,6,\quad |C_1^4-C_2^4|\leq\frac{C\,\mathrm{cof}(A)_{44}}{\det{A}}\leq C.
	\end{equation*}
	Proposition \ref{lemCialpha3D} is proved.
\end{proof}

\subsection{Proof of the Improved Estimates on $p_i^\alpha$}\label{sec4.4}

For $(i,\alpha)\neq (1,3)$, we only prove the improved estimate \eqref{improveof_p} in Proposition \ref{propu113D1}, since the others are the same.
\begin{proof}[Proof of \eqref{improveof_p} in Proposition \ref{propu113D1}]
	For $({\bf u}_1^1,p_1^1)$, in view of \eqref{equ_v12D}, since $\nabla\cdot{\bf u}_1^1=0$, it follows that for  $x\in B_{\delta(x')/4}$, and a fixed point $(z',0)$ with $|z'|=R$,
	\begin{align*}
		\Delta\big(p_i^\alpha(x)-p_i^\alpha(z',0)\big)=0.
	\end{align*}
	By using mean value property, we have
	\begin{align}\label{pgj3}
		|p_i^\alpha(x',0)-p_i^\alpha(z',0)|=&\,\Big|\fint_{B_{\delta(x')/4}(x',0)} p_i^\alpha-p_i^\alpha(z',0)\Big|\nonumber\\\leq&\, \Big|\fint_{B_{\delta(x')/4}(x',0)} (p_i^\alpha-p_i^\alpha(z',0)-\overline{p}_i^\alpha)\Big|+\Big|\fint_{B_{\delta(x')/4}(x',0)}\overline{p}_i^\alpha\Big|\nonumber\\
		\leq&\, C\fint_{B_{\delta(x')/4}(x',0)}|\nabla q_i^\alpha| +\Big|\fint_{B_{\delta(x')/4}(x',0)}\overline{p}_1^\alpha\Big|+C\nonumber\\
		:=&\,\mbox{III}_1+\mbox{III}_2+C,
	\end{align}
where $q_{1}^{\alpha}:=p_{1}^{\alpha}-\overline{p}_{1}^{\alpha}$.
By virtue of \cite[Proposition 2.1]{LX2} 
\begin{align}\label{LX_p}
		|p_{i}^{\alpha}(x)-(q_{i}^\alpha)_{R}|\leq\frac{C}{\varepsilon},\quad|\nabla p_{i}^{\alpha}(x)|\leq C\Big(\frac{1}{\delta(x')}+\frac{|x'|}{\delta(x')^2}\Big),\quad x\in\Omega_{R},
\end{align}
where $(q_{i}^\alpha)_{R}:=\frac{1}{|\Omega_{R}\setminus\Omega_{R/2}|}\int_{\Omega_{R}\setminus\Omega_{R/2}}q_{i}^\alpha(x)\mathrm{d}x$, which is independent of $\varepsilon$. Hence $|\mbox{III}_1|\leq \frac{1}{\varepsilon}$. Recalling the definition of $\overline{p}_1^1$, it is easy to see that it is an odd function with respect to $x_1$ or $x_2$, hence we have
	\begin{align}\label{pgj1}
		|\mbox{III}_2|\leq C.
	\end{align}
	Therefore, by making use of \eqref{pgj3}--\eqref{pgj1}, we get
	\begin{align}\label{pgj4}
		|p_i^\alpha(x',0)-p_i^\alpha(z',0)|\leq \frac{1}{\varepsilon}.
	\end{align}
	Finally, by means of mean value theorem, \eqref{pgj4} and  \eqref{LX_p} again, we derive
	\begin{align*}
		|p_i^\alpha(x)-p_i^\alpha(z',0)|&=|p_i^\alpha(x)-p_i^\alpha(x',0)|+|p_i^\alpha(x',0)-p_i^\alpha(z',0)|\\
		\leq&\, \delta(x')|\nabla p_i^\alpha|+|p_i^\alpha(x',0)-p_i^\alpha(z',0)|\leq 
			\frac{1}{\varepsilon}.
	\end{align*}
	
For $p_1^3$, we can derive that
	\begin{equation}\label{pgj5}
		|\widetilde{p}_1^3(x)-\widetilde{p}_1^3(z',0)|\leq \frac{C}{\varepsilon^2},
	\end{equation}
	where $(\widetilde{\bf u}_1^3, \widetilde{p}_1^3)$ satisfy the following linear equations,
	\begin{equation}\label{?1equ_v12D}
		\begin{cases}
			\nabla\cdot\sigma[\widetilde{\bf u}_{1}^3,\widetilde{p}_{1}^{3}]=0,\quad\nabla\cdot \widetilde{\bf u}_{1}^{3}=0,&\mathrm{in}~\Omega,\\
			\widetilde{\bf u}_{1}^{3}={\boldsymbol\psi}_{3},&\mathrm{on}~\partial{D}_{1},\\
			\widetilde{\bf u}_{1}^{3}=0,&\mathrm{on}~\partial{D_{2}}\cup\partial{D},
		\end{cases}
	\end{equation}

Denote 
	$$\overline{\bf u}_1^3={\bf u}_1^3-\widetilde{\bf u}_1^3,~\overline{p}_1^3=p_1^3-\widetilde{p}_1^3,$$
	it follows from \eqref{equ_v13} and \eqref{?1equ_v12D} that
	\begin{equation*}
		\begin{cases}
			\nabla\cdot\sigma[\overline{\bf u}_{1}^2,\overline{p}_{1}^{2}]={\bf u}\cdot\nabla{\bf u},\quad\nabla\cdot \overline{\bf u}_{i}^{\alpha}=0,&\mathrm{in}~\Omega,\\
			\overline{\bf u}_{i}^{\alpha}=0,&\mathrm{on}~\partial{\Omega}.
		\end{cases}
	\end{equation*}
By using Proposition \ref{lemma11} and the estimates for ${\bf u}_i^\alpha$ in Propositions \ref{propu113D1}, \ref{propu4563D}, \ref{propu133D}, we have 
	\begin{align}\label{?xjfz1}
		|{\bf u}\cdot\nabla{\bf u}|&\leq \Big| \Big(C_1^{3}{\bf u}_{1}^{3}+{\bf u}^{\#3}\Big)\cdot\nabla\Big(C_1^{3}{\bf u}_{1}^{3}+{\bf u}^{\#3}\Big)\Big|\nonumber\\
		&\leq \Big|\Big(C_1^{3}{\bf w}_{1}^{3}+C_1^{3}{\bf v}_{1}^{3}+{\bf w}^{\#3}+{\bf v}^{\#3}\Big)\cdot\nabla\Big(C_1^{3}{\bf w}_{1}^{3}+C_1^{3}{\bf v}_{1}^{3}+{\bf w}^{\#3}+{\bf v}^{\#3}\Big)\Big|\nonumber\\
		&\leq \sum_{i,j=1,2}|{\bf v}_i^3\cdot\nabla{\bf v}_j^3|+\frac{C|x'|}{\delta(x')^2}+\frac{C}{\delta(x')}\leq C\Big(\frac{|x'|}{\delta(x')^2}+\frac{1}{\delta(x')}\Big).
	\end{align}
By means of Proposition \ref{nsdpro} and \eqref{?xjfz1}, we deduce
	\begin{equation}\label{pfz1}
		|\overline{p}_1^3(x)-\overline{p}_1^3(z',0)|\leq C\|\nabla\overline{p}_1^3\|_{L^{\infty}(\Omega_{R})}\leq \frac{C}{\varepsilon^{3/2}}.
	\end{equation}
	Combining \eqref{pgj5} and \eqref{pfz1}, we have
	\begin{equation*}
		|{p}_1^3(x)-{p}_1^3(z',0)|\leq \frac{C}{\varepsilon^2}.
	\end{equation*}
	We thus complete the proof.
\end{proof}

\section{Proof of Theorems \ref{theolow}: Lower Bounds}\label{sec6}

Our proof is based on comparing the gradient of the solutions to the Navier-Stokes equations with that of the Stokes equations.
\begin{proof}[Proof of Theorem \ref{theolow}]
Under the same assumption as in Theorem \ref{theolow}, it had been proven in \cite[Theorem 1.4]{LX2} that,
\begin{align}\label{xj1}
|\nabla {\bf u}_{st} (0',x_3)|\ge \frac{1}{C|\log\varepsilon|{\varepsilon}},\quad |x_2|\leq \varepsilon,
\end{align}
where $( {\bf u}_{st}, p_{st})$ is the solution to the following Stokes equations
\begin{align}\label{Stokessys}
	\begin{cases}
		\nabla\cdot\sigma[{\bf u}_{st},p_{st}]=0,~~\nabla\cdot {\bf u}_{st}=0,&\hbox{in}~~\Omega,\\
		~{\bf u}_{st}|_{+}={\bf u}_{st}|_{-},&\hbox{on}~~\partial{D_{i}},~i=1,2,\\
		e({\bf u}_{st})=0, &\hbox{in}~~D_{i},~i=1,2,\\
		\int_{\partial{D}_{i}}\sigma[{\bf u}_{st},p_{st}]\cdot{\boldsymbol\psi}_{\alpha}
		\nu=0
		,&i=1,2,\alpha=1,2,\dots,6,\\
		{\bf u}_{st}={\boldsymbol\varphi}, &\hbox{on}~~\partial{D}.
	\end{cases}
\end{align}
We decompose the solution of \eqref{Stokessys} as follows:
\begin{equation*}
 {{\bf u}_{st}}=\sum_{\alpha=1}^{6}\left(\widetilde{C}_{1}^{\alpha}-\widetilde{C}_{2}^{\alpha}\right)\widetilde{\bf u}_{1}^{\alpha}
	+\nabla \widetilde{\bf u}_{b},~~p_{st}(x)=\sum_{i=1}^{2}\sum_{\alpha=1}^{6}\widetilde{C}_i^{\alpha}\widetilde{p}_{i}^{\alpha}(x)+\widetilde{p}_{0}(x),
\end{equation*}
where
$ 	\widetilde{\bf u}_{b}:=\sum_{\alpha=1}^{6}\widetilde{C}_{2}^{\alpha}(\widetilde{\bf u}_{1}^{\alpha}+\widetilde{\bf u}_{2}^{\alpha})+\widetilde{\bf u}_{0},
$ and 
 $\widetilde{\bf u}_{i}^{\alpha},\widetilde{\bf u}_{0}\in{C}^{2}(\Omega;\mathbb R^3),~\widetilde{p}_{i}^{\alpha}, \widetilde{p}_0\in{C}^{1}(\Omega)$, respectively, satisfy
\begin{equation}\label{1equ_v12D}
	\begin{cases}
		\nabla\cdot\sigma[\widetilde{\bf u}_{i}^\alpha,\widetilde{p}_{i}^{\alpha}]=0,\quad\nabla\cdot \widetilde{\bf u}_{i}^{\alpha}=0,&\mathrm{in}~\Omega,\\
		\widetilde{\bf u}_{i}^{\alpha}={\boldsymbol\psi}_{\alpha},&\mathrm{on}~\partial{D}_{i},\\
		\widetilde{\bf u}_{i}^{\alpha}=0,&\mathrm{on}~\partial{D_{j}}\cup\partial{D},~j\neq i,
	\end{cases}i=1,2,
\end{equation}
and
\begin{equation}\label{1equ_u02D}
	\begin{cases}
		\nabla\cdot\sigma[\widetilde{\bf u}_{0},\widetilde{p}_0]=0,\quad\nabla\cdot \widetilde{\bf u}_{0}=0,&\mathrm{in}~\Omega,\\
		\widetilde{\bf u}_{0}=0,&\mathrm{on}~\partial{D}_{1}\cup\partial{D_{2}},\\
		\widetilde{\bf u}_{0}={\boldsymbol\varphi},&\mathrm{on}~\partial{D}.
	\end{cases}
\end{equation}

Denote 
$$\widehat{\bf u}_i^\alpha={\bf u}_i^\alpha-\widetilde{\bf u}_i^\alpha,~\widehat{p}_1^\alpha=p_i^\alpha-\widetilde{p}_i^\alpha$$
 it follows from \eqref{equ_v1}, \eqref{equ_u0}, \eqref{equ_v12Dd}, \eqref{1equ_v12D} and \eqref{1equ_u02D} that, for $(i,\alpha)\neq (1,3)$, $\overline{\bf u}_i^\alpha=0$ and 
 \begin{equation*}
 	\begin{cases}
 		\nabla\cdot\sigma[\widehat{\bf u}_{1}^3,\widehat{p}_{1}^{3}]+{\bf u}\cdot\nabla{\bf u}=0,\quad\nabla\cdot \widehat{\bf u}_{i}^{\alpha}=0,&\mathrm{in}~\Omega,\\
 		\widehat{\bf u}_{i}^{\alpha}=0,&\mathrm{on}~\partial{\Omega}.
 	\end{cases}
 \end{equation*}
 By using Theorem \ref{mainthm} and \eqref{zydfzgs11} we have
 \begin{align}\label{xjfz1}
 |{\bf u}(x)\cdot\nabla{\bf u}(x)|\leq \frac{C}{|\log\varepsilon|{\delta(x')}},\quad \,x\in\Omega_{R}.
 \end{align}
In view of Proposition \ref{nsdpro} and \eqref{xjfz1}, we deduce
\begin{align}\label{xjfz2}
|\nabla \widehat{\bf u}_1^3(x)|\leq \frac{C}{|\log\varepsilon|},\quad \,x\in\Omega_{R}.
\end{align}
Recalling the decomposition of the solution in \eqref{udecom}, by making use of Propositions \ref{lemma11}, \ref{prop1.7} and \eqref{xjfz2}, we derive
\begin{align}\label{xjfz3}
| \nabla({\bf u}-{{\bf u}_{st}})(x)|&=\Big|\sum_{\alpha=1}^{6}\left({C}_{1}^{\alpha}-{C}_{2}^{\alpha}-(\widetilde{C}_{1}^{\alpha}-\widetilde{C}_{2}^{\alpha})\right)\widehat{\bf u}_{1}^{\alpha}(x)
+ \overline{\bf u}_{b}(x)\Big|\nonumber\\
&=\left|{C}_{1}^{3}-{C}_{2}^{3}-(\widetilde{C}_{1}^{3}-\widetilde{C}_{2}^{3})\right||\nabla\widehat{\bf u}_{1}^{3}(x)|+C\leq C,\quad \,x\in\Omega_{R}.
\end{align}
By using \eqref{xj1} and \eqref{xjfz3}, we get, for $|x_3|\leq \varepsilon$,
\begin{align*}
|\nabla{\bf u} (0',x_3)|\ge |\nabla {\bf u}_{st} (0',x_3)|-| \nabla({\bf u}-{{\bf u}_{st}})(0',x_3)|\ge \frac{1}{C|\log\varepsilon|{\varepsilon}}-C\ge \frac{1}{C|\log\varepsilon|{\varepsilon}}.
\end{align*}
Thus, Theorem \ref{theolow} is proved.
\end{proof}

 \section{Main Results in 2D}\label{sec5}
Our approach works well for the problem in two dimensions. We have
\begin{theorem}\label{mainthm2D}(Upper Bounds in 2D)
 	Assume that $D_1,D_2,D$, and $\varepsilon$ are defined as in Subsection \ref{subsec1.1}, with given ${\boldsymbol\varphi}\in C^{2,\alpha}(\partial D;\mathbb R^2)$ for some $0<\alpha<1$. Let ${\bf u}\in W^{1,2}(D;\mathbb R^2)\cap C^1(\bar{\Omega};\mathbb R^2)$ and $p\in L^2(D)\cap C^0(\bar{\Omega})$ be the solution to \eqref{sto}-\eqref{compatibility}. Then for sufficiently small $0<\varepsilon<1$, the following assertions hold:
 	
 	(i) For $x\in \Omega_{R}$, 
 	\begin{align*}
 		&|\nabla{\bf u}(x)|\leq C\frac{\sqrt{\varepsilon}+|x_1|}{\varepsilon+x_1^2}\|{\boldsymbol\varphi}\|_{C^{2,\alpha}(\partial D)},\quad~\inf_{c\in\mathbb{R}}\|p+c\|_{C^0(\bar{\Omega}_{R})}\leq \frac{C}{\sqrt{\varepsilon}}\|{\boldsymbol\varphi}\|_{C^{2,\alpha}(\partial D)},
 	\end{align*}
 	and $\|\nabla{\bf u}\|_{L^{\infty}(\Omega\setminus\Omega_{R})}+\inf_{c\in\mathbb{R}}\|p+c\|_{L^{\infty}(\Omega\setminus\Omega_{R})}\leq\,C\|{\boldsymbol\varphi}\|_{C^{2,\alpha}(\partial D)},$ where $C$ is a universal constant. In particular, 
 	\begin{equation*}
 		\|\nabla{\bf u}\|_{L^{\infty}(\Omega)}\leq \frac{C}{\sqrt\varepsilon}\|{\boldsymbol\varphi}\|_{C^{2,\alpha}(\partial D)};
 	\end{equation*}
 	
 	(ii) For $x\in \Omega_{R}$, 
 	\begin{align*}
 		|\nabla^2{{\bf u}}(x)|+|\nabla p(x)|\leq\frac{C}{(\varepsilon+x_1^2)^{3/2}}\|{\boldsymbol\varphi}\|_{C^{2,\alpha}(\partial D)},
 	\end{align*}
 	and $\|\nabla^2{\bf u}\|_{L^{\infty}(\Omega\setminus\Omega_{R})}+\|\nabla p\|_{L^{\infty}(\Omega\setminus\Omega_{R})}\leq\,C\|{\boldsymbol\varphi}\|_{C^{2,\alpha}(\partial D)}$, where $C$ is a universal constant.
 \end{theorem}
 
 \begin{remark}
 	Here the blow up rate of $\nabla{\bf u}$ in dimension two is $\varepsilon^{-1/2}$, which is also optimal by Theorem \ref{theolow2d} below. 
 \end{remark}
The proof is very similar with the three dimensional case. We only present the main differences in this section. In dimension two, 
 $$\Psi=\mathrm{span}\Bigg\{{\boldsymbol\psi}_{1}=\begin{pmatrix}
 	1 \\
 	0
 \end{pmatrix},
 {\boldsymbol\psi}_{2}=\begin{pmatrix}
 	0\\
 	1
 \end{pmatrix},
 {\boldsymbol\psi}_{3}=\begin{pmatrix}
 	x_{2}\\
 	-x_{1}
 \end{pmatrix}
 \Bigg\}.$$
The solution of \eqref{sto} is decomposed as follows:
 \begin{align}
 	{\bf u}(x)&=\sum_{i=1}^{2}\sum_{\alpha=1}^{3}C_i^{\alpha}{\bf u}_{i}^{\alpha}(x)+{\bf u}_{0}(x),~
 	p(x)=\sum_{i=1}^{2}\sum_{\alpha=1}^{3}C_i^{\alpha}p_{i}^{\alpha}(x)+p_{0}(x),~ x\in\,\Omega,\label{ud}
 \end{align}
where we put the nonlinear term into the equation of $({\bf u}_1^2,p_1^2)$, such that $({\bf u}_1^2,p_1^2)$ satisfy the Navier-Stokes equations
  \begin{equation}\label{equ_v12Dd}
 	\begin{cases}
 		\nabla\cdot\sigma[{\bf u}_{1}^2,p_{1}^{2}]={\bf u}\cdot\nabla{\bf u},\quad\nabla\cdot {\bf u}_{1}^{2}=0,&\mathrm{in}~\Omega,\\
 		{\bf u}_{1}^{2}={\boldsymbol\psi}_{2},&\mathrm{on}~\partial{D}_{1},\\
 		{\bf u}_{1}^{2}=0,&\mathrm{on}~\partial{D_{2}}\cup\partial{D},
 	\end{cases}
 \end{equation}
 and for $(i,\alpha)\in \{(a,b)|a=1,2,b=1,2,3\}\setminus (1,2)$, ${\bf u}_{i}^{\alpha},{\bf u}_{0}\in{C}^{2}(\Omega;\mathbb R^2),~p_{i}^{\alpha}, p_0\in{C}^{1}(\Omega)$, respectively, satisfy, 
 \begin{equation}\label{equ_v12D}
 	\begin{cases}
 		\nabla\cdot\sigma[{\bf u}_{i}^\alpha,p_{i}^{\alpha}]=0,\quad\nabla\cdot {\bf u}_{i}^{\alpha}=0,&\mathrm{in}~\Omega,\\
 		{\bf u}_{i}^{\alpha}={\boldsymbol\psi}_{\alpha},&\mathrm{on}~\partial{D}_{i},\\
 		{\bf u}_{i}^{\alpha}=0,&\mathrm{on}~\partial{D_{j}}\cup\partial{D},~j\neq i,
 	\end{cases}i=1,2,
 \end{equation}
 and
 \begin{equation}\label{equ_u02D}
 	\begin{cases}
 		\nabla\cdot\sigma[{\bf u}_{0},p_0]=0,\quad\nabla\cdot {\bf u}_{0}=0,&\mathrm{in}~\Omega,\\
 		{\bf u}_{0}=0,&\mathrm{on}~\partial{D}_{1}\cup\partial{D_{2}},\\
 		{\bf u}_{0}={\boldsymbol\varphi},&\mathrm{on}~\partial{D}.
 	\end{cases}
 \end{equation}
 
Here we only consider the special case that $C_1^2=1$ and rewrite
 \begin{align*}
 	{\bf u}(x):=C_1^2{\bf u}_1^2(x)+{\bf u}^{\#2}(x),~\text{and}~p(x):=C_1^2\overline{p}_1^2(x)+p^{\#2}(x).
 \end{align*}
Then, the Navier-Stokes equation \eqref{equ_v12Dd} can be written as 
\begin{equation}\label{equ_v12Ddd}
	\left\{\begin{split}
		\mu \Delta{\bf u}_{1}^2=&\,\nabla p_{1}^{2}+(C_1^{2})^{2}{\bf u}_{1}^{2}\cdot\nabla\,{\bf u}_{1}^{2}+C_1^{2}{\bf u}_{1}^{2}\cdot\nabla{\bf u}^{\#2}\\
		&\quad\quad\quad\quad+C_1^{2}{\bf u}^{\#2}\cdot\nabla\,{\bf u}_{1}^{2}+{\bf u}^{\#2}\cdot\nabla{\bf u}^{\#2},\quad&\mathrm{in}&~\Omega,\\
		\nabla\cdot {\bf u}_{1}^{2}=&\,0,\quad\quad&\mathrm{in}&~\Omega,\\
		{\bf u}_{1}^{2}=&\,{\boldsymbol\psi}_{2},\quad\mathrm{on}~\partial{D}_{1},\quad
		{\bf u}_{1}^{2}=\,0,\quad\mathrm{on}~\partial{D_{2}}\cup\partial{D}.
	\end{split}\right.
\end{equation}
Thus, to estimate $({\bf u}_1^2,p_1^2)$, we need the estimates of $C_i^\alpha$ and of $({\bf u}_i^\alpha,p_i^\alpha)_{(i,\alpha)\neq(1,2)}$ and $({\bf u}_0,p_0)$. By using the trace theorem and the fact that $\|{\bf u}\|_{W^{1,2}(\Omega\setminus\Omega_R)}\leq C$, we have (c.f. \cite{BLL}) 
 \begin{prop}\label{lemma1}
 Let $C_i^\alpha$ be defined in \eqref{ud}. Then
 $$|C_i^\alpha|\leq C,~i=1,2;~~\alpha=1,2,3.$$
 \end{prop}

The Keller-type function $k(x)\in C^3(\mathbb{R}^2)$ is
\begin{align*}
	k(x)=\frac{x_2}{\delta(x_{1})}, ~\text{in}~\Omega_{2R},\quad\mbox{where}~\delta(x_{1}):=\epsilon+h_{1}(x_1)+h_{2}(x_1).
\end{align*}
For simplicity we assume that $h_1(x_1)=h_2(x_1)=\frac{1}{2}x_1^2$ for $|x_1|\leq 2R$. We construct ${\bf v}_{1}^{1}\in C^{2}(\Omega;\mathbb R^2)$, such that ${\bf v}_1^1={\bf u}_{1}^1=\boldsymbol{\psi}_{1}$ on $\partial{D}_{1}$ and ${\bf v}_1^1={\bf u}_{1}^1=0$ on $\partial{D}_{2}\cup\partial{D}$. Namely, 
 \begin{align}\label{v11}
 	{\bf v}_{1}^{1}=\Big(k(x)+\frac{1}{2}\Big)\boldsymbol\psi_{1}+x_{1}\Big(k(x)^2-\frac{1}{4}\Big)\boldsymbol\psi_{2},\quad\hbox{in}\ \Omega_{2R},
 \end{align}
satisfying $\nabla \cdot {\bf v}_i=0$ in $\Omega_{2R}$, and choose
$\overline{p}_1^1=\frac{2\mu x_1}{\delta(x_{1})}k(x)$, in $\Omega_{2R}$.
Then 
 \begin{prop}\label{propu11}(An Improvement of \cite[Proposition 2.1]{LX})
 	Let ${\bf u}_{i}^{1}\in{C}^{2}(\Omega;\mathbb R^2),~p_{i}^{1}\in{C}^{1}(\Omega)$ be the solution to \eqref{equ_v12D}. Then
 	\begin{equation*}
 		\|\nabla({\bf u}_{i}^{1}-{\bf v}_{i}^{1})\|_{L^{\infty}(\Omega_{\delta(x_1)/2}(x_1))}\leq C,\quad \,x\in\Omega_{R},
 	\end{equation*}
 	and
 	\begin{equation*}
 		\|\nabla^2({\bf u}_{i}^{1}-{\bf v}_{i}^{1})\|_{L^{\infty}(\Omega_{\delta(x_1)/2}(x_1))}+\|\nabla q_i^1\|_{L^{\infty}(\Omega_{\delta(x_1)/2}(x_1))}\leq \frac{C}{\delta(x_{1})},~~ \,x\in\Omega_{R}.
 	\end{equation*}
 	Consequently, 
 	\begin{align*}
 		\frac{1}{C\delta(x_{1})}\leq|\nabla {\bf u}_{i}^1(x)|\leq \frac{C}{\delta(x_{1})},~|\nabla^2{\bf u}_{i}^1(x)|\leq  C\Big(\frac{1}{\delta(x_{1})}+\frac{|x_1|}{\delta(x_{1})^2}\Big),\quad\,x\in\Omega_{R},
 	\end{align*}
 	and 
 	$$|p_{i}^{1}(x)-p_i^1(R,0)|\leq\frac{C}{\varepsilon},~|\nabla p_{i}^{1}(x)|\leq C\Big(\frac{1}{\delta(x_{1})}+\frac{|x_1|}{\delta(x_{1})^2}\Big),\quad\,x\in\Omega_{R}.$$
 \end{prop}

 For ${\bf u}_{i}^{3}$, we construct  ${\bf v}_{1}^{3}\in C^{2}(\Omega;\mathbb R^2)$, such that, in $\Omega_{2R}$, 
 \begin{align*}
 	{\bf v}_{1}^{3}=\boldsymbol\psi_{3}\Big(k(x)+\frac{1}{2}\Big)
 	+\begin{pmatrix}
 		1-\frac{4x_{1}^2}{\delta(x_{1})}-5x_{2}k(x)\\\\
 		2x_1k(x)\left(2-\frac{4x_{1}^{2}}{\delta(x_{1})}-3x_{2}k(x)\right)
 	\end{pmatrix}
 	\Big(k(x)^2-\frac{1}{4}\Big),
 \end{align*}
 and choose $\overline{p}_1^3=\frac{2\mu x_1}{\delta(x_{1})^2}+\frac{12\mu x_1 }{\delta(x_{1})}\Big(1-\frac{2 x_1^2}{\delta(x_{1})}\Big)k(x)^{2}$.
Then
 \begin{prop}\label{propu13}(An Improvement of \cite[Proposition 2.3]{LX})
 	Let ${\bf u}_{i}^{3}\in{C}^{2}(\Omega;\mathbb R^2),~p_{i}^{3}\in{C}^{1}(\Omega)$ be the solution to \eqref{equ_v12D}. Then 
 	\begin{equation*}
 		\|\nabla({\bf u}_{i}^{3}-{\bf v}_{i}^{3})\|_{L^{\infty}(\Omega_{\delta(x_1)/2}(x_1))}\leq C,\quad x\in\Omega_{R},
 	\end{equation*}
 	and 
 	\begin{equation*}
 		\|\nabla^2({\bf u}_{i}^{3}-{\bf v}_{i}^{3})\|_{L^{\infty}(\Omega_{\delta(x_1)/2}(x_1))}+\|\nabla q_{i}^{3}\|_{L^{\infty}(\Omega_{\delta(x_1)/2}(x_1))}\leq \frac{C}{\delta(x_{1})},\quad x\in\Omega_{R}.
 	\end{equation*}
 	Consequently,
 	\begin{equation*}
 		|\nabla {\bf u}_{i}^3(x)|\leq \frac{C}{\delta(x_{1})},~|\nabla^2 {\bf u}_{i}^3(x)|\leq\frac{C}{\delta(x_{1})^2},\quad\,x\in\Omega_{R},
 	\end{equation*}
 	and
 	\begin{equation*}
 		|p_{i}^{3}(x)-p_i^3(R,0)|\leq\frac{C}{\varepsilon},~|\nabla p_{i}^{3}(x)|\leq\frac{C}{\delta(x_1)^{2}},\quad\,x\in\Omega_{R}.
 	\end{equation*}
 \end{prop}

Similarly as in Proposition \ref{prop1.7},  we have the following results.
\begin{prop}\label{propu0} (\cite{LX})
	Let ${\bf u}_0, {\bf u}_i^\alpha\in{C}^{2}(\Omega;\mathbb R^2),~p_0,p_i^\alpha\in{C}^{1}(\Omega)$ be the solution to \eqref{equ_u02D} and  \eqref{equ_v12D}. Then, for $\alpha\neq 2$, we have
\begin{equation*}
	\|\nabla^{k_1}{\bf u}_0\|_{L^{\infty}(\Omega)}+\|\nabla^{k_1}({\bf u}_{1}^{\alpha}+{\bf u}_{2}^{\alpha})\|_{L^{\infty}(\Omega)}\leq C,\quad k_1=1,2,
\end{equation*}
and
\begin{equation*}
	\|\nabla^{k_2} p_0\|_{L^{\infty}(\Omega)}+\|\nabla^{k_2}(p_{1}^{\alpha}+p_{2}^{\alpha})\|_{L^{\infty}(\Omega)}\leq C,\quad k_2=0,1.
\end{equation*}
\end{prop}

For $\alpha=2$, similarly as in Proposition \ref{propuhe111}, 
\begin{prop}\label{propuhe11}
	Let ${\bf u}_2^2, {\bf u}_1^2\in{C}^{2}(\Omega;\mathbb R^2),~p_2^2, p_1^2\in{C}^{1}(\Omega)$ be the solution to \eqref{equ_v12D} and \eqref{equ_v12Dd}. Then
	\begin{equation*}
		\|\nabla({\bf u}_{1}^{2}+{\bf u}_{2}^{2})\|_{L^{\infty}(\Omega_{\delta(x_1)/2}(x_1))}\leq \frac{C|C_1^2-C_2^2|}{\sqrt{\delta(x_{1})}}+C,
	\end{equation*}
and 
	\begin{align*}
	&\|\nabla^2({\bf u}_{1}^{2}+{\bf u}_{2}^{2})\|_{L^{\infty}(\Omega_{\delta(x_1)/2}(x_1))}+\|\nabla(p_1^2+p_2^2)\|_{L^{\infty}(\Omega_{\delta(x_1)/2}(x_1))}\\
	\leq&\,\frac{C|C_1^2-C_2^2||x_1|}{\delta(x_{1})^2}+\frac{C\sum_{\alpha=1}^{3}|C_{1}^{\alpha}-C_{2}^{\alpha}|}{\delta(x_{1})}+C.
\end{align*}
\end{prop}
To estimate $\nabla{\bf u}_{1}^{2}$, we  construct  ${\bf v}_{1}^{2}\in C^{2}(\Omega;\mathbb R^2)$, such that ${\bf v}_1^2={\bf u}_{1}^2=\boldsymbol{\psi}_{2}$ on $\partial{D}_{1}$ and ${\bf v}_1^2={\bf u}_{1}^2=0$ on $\partial{D}_{2}\cup\partial{D}$.  
\begin{align*}
	{\bf v}_{1}^{2}=\boldsymbol\psi_{2}\Big(k(x)+\frac{1}{2}\Big)
	+\frac{6}{\delta(x_{1})}\Big(\boldsymbol\psi_{1}x_1+\boldsymbol\psi_{2}x_{2}\big(\frac{2x_1^2}{\delta(x_{1})}-\frac{1}{3}\big)\Big)
	\Big(k(x)^2-\frac{1}{4}\Big),~\mbox{in}~\Omega_{2R},
\end{align*}
satisfying $\nabla \cdot {\bf v}_1^2=0$ in $\Omega_{2R}$, and $\|{\bf v}_{1}^{2}\|_{C^{3}(\Omega\setminus\Omega_{R})}\leq\,C$. 
Moreover, we choose a $\overline{p}_1^2\in C^{1}(\Omega)$ such that
\begin{equation*}
	\overline{p}_1^2=-\frac{3\mu}{\delta(x_{1})^2}+\frac{18\mu}{\delta(x_{1})}\Big(\frac{2x_{1}^{2}}{\delta(x_{1})}-\frac{1}{3}\Big)k(x)^{2},\quad\mbox{in}~\Omega_{2R},
\end{equation*}
and  $\|\overline{p}_1^2\|_{C^{1}(\Omega\setminus\Omega_{R})}\leq C$. 

It follows from \cite[Proposition 2.2]{LX} that the following Proposition holds for $({{\bf u}}_2^2,p_2^2)$, since $({{\bf u}}_2^2,p_2^2)$ is the solution to Stokes flow. Using the estimates  of ${\bf u}_i^\alpha$, $(i,\alpha)\neq (1,2)$ obtained before, we prove Proposition \ref{propu12} holds as well for  $({{\bf u}}_1^2,p_1^2)$.

\begin{prop}\label{propu12}
	Let ${\bf u}_{1}^{2}, {\bf u}_{2}^{2}\in{C}^{2}(\Omega;\mathbb R^2),~p_{1}^{2}, p_{2}^{2}\in{C}^{1}(\Omega)$ be the solution to \eqref{equ_v12Dd} and \eqref{equ_v12D}. Then we have
	\begin{equation*}
		\|\nabla({\bf u}_{i}^{2}-{\bf v}_{i}^{2})\|_{L^{\infty}(\Omega_{\delta(x_1)/2}(x_1))}\leq \frac{C}{\sqrt{\delta(x_{1})}},\quad\,x\in\Omega_{R},
	\end{equation*}
	and 
	\begin{equation*}
		\|\nabla^2({\bf u}_{i}^{2}-{\bf v}_{i}^{2})\|_{L^{\infty}(\Omega_{\delta(x_1)/2}(x_1))}+\|\nabla q_i^2\|_{L^{\infty}(\Omega_{\delta(x_1)/2}(x_1))}\leq C\Big(\frac{1}{\delta(x_{1})}+\frac{|x_1|}{\delta(x_{1})^2}\Big),\quad\,x\in\Omega_{R}.
	\end{equation*}
	Consequently, for $x\in\Omega_{R}$,
	\begin{align*}
		\frac{1}{C\delta(x_{1})}\leq|\nabla {\bf u}_{i}^2(x)|\leq C\Big(\frac{1}{\delta(x_{1})}+\frac{|x_1|}{\delta(x_{1})^2}\Big),~|\nabla^2 {\bf u}_{i}^2(x)|\leq C\Big(\frac{1}{\delta(x_{1})^2}+\frac{|x_1|}{\delta(x_1)^3}\Big),
	\end{align*}
	and
	$$|p_{i}^{2}(x)-p_i^2(R,0)|\leq\,\frac{C}{\varepsilon^2},~|\nabla p_{i}^{2}(x)|\leq\,C\Big(\frac{1}{\delta(x_{1})^2}+\frac{|x_1|}{\delta(x_1)^3}\Big).$$
\end{prop}

Denote
\begin{align*}
	{\bf w}_{2}:={\bf u}_1^{2}-{{\bf v}}_{1}^2,\quad\mbox{and}~ q_{2}:=p_1^{2}-\overline{p}_1^{2}.
\end{align*} 
Then $({\bf w}_2,q_2)$ verify the following boundary value problem
\begin{align*}
	\begin{cases}
		\mu\,\Delta {\bf w}_2=\nabla q_2+(C_1^2)^2{\bf w}_2\cdot\nabla {\bf w}_2+C_1^2\widetilde{\bf v}\cdot\nabla {\bf w}_2+ C_1^2{\bf w}_2\cdot\nabla\widetilde{\bf v}-{\bf f},&\mathrm{in}\;\Omega,\\
		\nabla\cdot{\bf{w}}_{2}=0,&\mathrm{in}\;\Omega_{R},\\
		|\nabla {\bf{w}}_{2}|\leq\,C,&\mathrm{in}\;\Omega\setminus\Omega_{R},\\
		{\bf{w}}_{2}=0,&\mathrm{on}\;\partial \Omega.
	\end{cases}
\end{align*}
where ${\bf f}=\mu\,\Delta {\bf v}_1^2-\nabla \overline{p}_1^2-C_1^2\widetilde{\bf v}_1^2\cdot \nabla {\bf v}_1^2-\widetilde{\bf v}\cdot \nabla {\bf u}^{\#2},~\text{and}~\widetilde{\bf v}= C_1^2{\bf v}_1^2 + {\bf u}^{\#2}.$
Apply Lemma \ref{lemmaenergy} and the iteration process in Lemma \ref{nszyly333}, we have $\int_{\Omega_\delta(z_1)}|\nabla{\bf w}_2|^2\leq C\delta(z_1)^2$. Using
$|{\bf f}(x)|\leq \,C\Big(\frac{|x_1|}{\delta(x_{1})^2}+\frac{1}{\delta(x_{1})}\Big)$, and \eqref{Wmpstokes}, we can prove Proposition \ref{propu12} holds.

\subsection{Estimates of $|C_1^\alpha-C_2^\alpha|$}\label{sec5.4}
By \eqref{sto-2} and decomposition \eqref{udecom}, instead of \eqref{equ-decompositon?}, we have
\begin{equation*}
	\sum_{i=1}^2\sum\limits_{\alpha=1}^{3} C_{i}^{\alpha}
	\int_{\partial D_j}{\boldsymbol\psi}_\beta\cdot\sigma[{\bf u}_{i}^\alpha,p_{i}^{\alpha}]\nu
	+\int_{\partial D_j}{\boldsymbol\psi}_\beta\cdot\sigma[{\bf u}_{0},p_{0}]\nu=0,~~\beta= 1,2,3,
\end{equation*}
where $j=1,2$. Similarly as \eqref{systemC} in 3D, we have
\begin{align*}
	\begin{cases}
		\sum\limits_{\alpha=1}^{3}C_{1}^{\alpha}a_{11}^{\alpha\beta}
		+\sum\limits_{\alpha=1}^{3}C_{2}^{\alpha}a_{21}^{\alpha\beta}
		-b_{1}^{\beta}=0,&\\
		\sum\limits_{\alpha=1}^{3}C_{1}^{\alpha}a_{12}^{\alpha\beta}
		+\sum\limits_{\alpha=1}^{3}C_{2}^{\alpha}a_{22}^{\alpha\beta}
		-b_{2}^{\beta}=0.
	\end{cases}\quad \beta=1,2,3.
\end{align*}
\subsection{Some Technical Results in 2D}

It is different with the 3D case. To estimate $a_{1j}^{2\beta}$, we have to estimate the trilinear form ${\bf T}({\bf u},{\bf u},{\bf u}_j^\beta)$, for  $\beta=1,2,3$, where  
$${\bf T}({\bf u},{\bf v},{\bf w})=\int_{\Omega}{\bf u}\cdot\nabla{\bf v}\cdot{\bf w},$$
because ${\bf T}({\bf u},{\bf u},{\bf u}_1^\beta)$ in dimension two may exhibit singularity. Specifically, 
\begin{prop}\label{ewdgj}
 We have
 \begin{equation*}
 |{\bf T}({\bf u},{\bf u},{\bf u}_1^\beta)|\leq C|C_1^2-C_2^2||\log\varepsilon|+C,~~\beta=1,3,
 \end{equation*}
and
\begin{equation*}
|{\bf T}({\bf u},{\bf u},{\bf u}_1^2)|\leq \frac{C|C_1^2-C_2^2|}{\sqrt{\varepsilon}}+C\,|\log\varepsilon|.
\end{equation*}
\end{prop}
\begin{proof}
	We rewrite \eqref{ud} as
	$${\bf u}=\sum_{\alpha=1}^{3}|C_1^\alpha-C_2^\alpha|{\bf u}_1^\alpha+{\bf u}_{b},~\text{and}~{\bf u}_{b}:=\sum_{\alpha=1}^{3}C_{2}^{\alpha}({\bf u}_{1}^{\alpha}+{\bf u}_{2}^{\alpha})+{\bf u}_{0}.$$
By using Propositions \ref{lemma1}--\ref{propu0} and the mean value theorem, we derive
	\begin{align}\label{dszfz1}
		|{\bf u}_b|=\Big|\sum_{\alpha=1}^{3}C_{2}^{\alpha}\big({\bf u}_{1}^{\alpha}+{\bf u}_{2}^{\alpha}-({\bf v}_{1}^{\alpha}+{\bf v}_{2}^{\alpha})\big)+{\bf u}_{0}\Big|+\Big|\sum_{\alpha=1}^{3}C_{2}^{\alpha}({\bf v}_{1}^{\alpha}+{\bf v}_{2}^{\alpha}\Big|\leq C,
	\end{align}
	where we use the fact that $|{\bf v}_{1}^{\alpha}+{\bf v}_{2}^{\alpha}|\leq C$, $\alpha=1,2,3.$ 
	Moreover, in view of Propositions \ref{propu0} and \ref{propuhe11}, we have
	\begin{align}\label{dszfz2}
	|\nabla{\bf u}_b|\leq C\sum_{\alpha=1}^{3}\nabla({\bf u}_{1}^{\alpha}+{\bf u}_{2}^{\alpha})+\nabla {\bf u}_0\leq \frac{C|C_1^2-C_2^2|}{\sqrt{\delta(x_{1})}}+C.
	\end{align}
 It follows from Propositions \ref{lemma1}--\ref{propu0}, Proposition \ref{propu12}, \eqref{dszfz1} and \eqref{dszfz2} that
\begin{align}\label{fz222}
	|{\bf u}\cdot\nabla{\bf u}|\leq&\, \Big|\Big(\sum_{\alpha=1}^{3}(C_{1}^{\alpha}-C_{2}^{\alpha})({\bf w}_{1}^{\alpha}+{\bf v}_{1}^{\alpha})
	+{\bf u}_{b}\Big)\cdot\Big(\sum_{\alpha=1}^{3}(C_{1}^{\alpha}-C_{2}^{\alpha})\nabla({\bf w}_1^\alpha+{\bf v}_{1}^{\alpha})
	+\nabla {\bf u}_{b}\Big)\Big|\nonumber\\
	 \leq &\, \frac{C|C_1^2-C_2^2|}{\delta(x_1)^{3/2}}+\frac{C}{\delta(x_{1})}.
\end{align}
In view of Proposition \ref{propu11}, Proposition \ref{propu13} and the mean value theorem, we have
\begin{align}\label{fz333}
|{\bf u}_1^\beta|=|{\bf u}_1^\beta-{\bf v}_1^\beta|+|{\bf v}_1^\beta|\leq C,~i=1,3,~|{\bf u}_1^2|=|{\bf u}_1^2-{\bf v}_1^2|+|{\bf v}_1^2|\leq \frac{C}{\sqrt{\delta(x_{1})}},
\end{align}
where we use the fact that $|{\bf v}_1^\beta|\leq C$, $i=1,3$ and  $|{\bf v}_1^2|\leq  \frac{C}{\sqrt{\delta(x_{1})}}$.
Thus, combining \eqref{fz222} and \eqref{fz333}, we deduce, for $i=1,3$,
\begin{align*}
 |{\bf T}({\bf u},{\bf u},{\bf u}_1^\beta)|\leq &\int_{\Omega_{R}}{\bf u}\cdot\nabla{\bf u}\cdot{\bf u}_1^\beta+C\leq C|C_1^2-C_2^2|\int_{\Omega_{R}}\frac{1}{\delta(x_1)^{3/2}}+C
 \leq C|C_1^2-C_2^2||\log\varepsilon|+C,
\end{align*}
and 
\begin{align*}
	|{\bf T}({\bf u},{\bf u},{\bf u}_1^2)|\leq& \int_{\Omega_{R}}{\bf u}\cdot\nabla{\bf u}\cdot{\bf u}_1^2+C
	\leq\, C|C_1^2-C_2^2|\int_{\Omega_{R}}\frac{1}{\delta(x_1)^{2}}+ \int_{\Omega_{R}}\frac{C}{\delta(x_1)^{3/2}}\\
	\leq&\, \frac{C|C_1^2-C_2^2|}{\sqrt{\varepsilon}}+C|\log\varepsilon|.
\end{align*}
Hence, the proof of Proposition \ref{ewdgj} is finished.
\end{proof}

\begin{lemma}\label{lema11} In dimension two,
	\begin{align}
		\frac{1}{C\sqrt{\varepsilon}}\leq a_{11}^{11}\leq&\, \frac{C}{\sqrt{\varepsilon}},\quad\quad~~
		\frac{1}{C\varepsilon^{3/2}}\leq a_{11}^{22}\leq \frac{C}{\varepsilon^{3/2}},\quad\quad
		~~\frac{1}{C\sqrt{\varepsilon}}\leq a_{11}^{33}\leq \frac{C}{\sqrt{\varepsilon}},\label{esta1111}\\
		|a_{11}^{12}|,~|a_{11}^{23}|\leq&\, C|\log\varepsilon|,\quad\quad\quad~|a_{11}^{13}|\leq C,\label{esta1211}\\
		~|a_{11}^{\alpha\beta}+a_{21}^{\alpha\beta}|\leq&\, C,\quad\quad~\alpha=1,3,\quad\beta=1,2,3,~\alpha\neq\beta,\label{esta1112}\\
		|a_{11}^{2\beta}+a_{21}^{2\beta}|\leq&\, C|C_1^2-C_2^2||\log\varepsilon|+C, \quad\quad~\beta=1,3,\label{esta111123}\\
			|a_{11}^{22}+a_{21}^{22}|\leq&\, \frac{C|C_1^2-C_2^2|}{\varepsilon}+\frac{C\sum_{\alpha=1}^3|C_1^\alpha-C_2^\alpha|}{\sqrt{\varepsilon}}+C,\label{esta11112}
	\end{align}
   and
	\begin{equation}\label{estb1}
		|b_1^\beta|\leq C,\quad\quad~\beta=1,2,3.
	\end{equation}
\end{lemma}
\begin{proof}
 We know from Proposition \ref{ewdgj} and \cite[Lemma 4.8]{LX} that \eqref{esta1111}--\eqref{esta1112} and \eqref{estb1} hold. For \eqref{esta111123}, recalling the definition of $a_{ij}^{\alpha\beta}$, we obtain
		\begin{align*}
		a_{11}^{21}+a_{21}^{21}=\int_{\Omega}\left(2\mu e({\bf u}_{1}^{2}+{\bf u}_2^2), e({\bf u}_{1}^{1})\right)+\int_{\Omega}{\bf u}\cdot \nabla{\bf u}\cdot {\bf u}_{1}^{1}:=\mbox{II}_1+\mbox{II}_2.
	\end{align*}
By using Proposition \ref{propuhe11}, we have
  \begin{align*}
  |\mbox{II}_1|\leq C|C_1^2-C_2^2|\int_{\Omega_{R}}\frac{1}{\delta(z')^{3/2}(x_1)}+C\leq C|C_1^2-C_2^2||\log\varepsilon|+C,
  \end{align*}
  and it follows from Proposition \ref{ewdgj} that
  \begin{align*}
  |\mbox{II}_2| \leq C|C_1^2-C_2^2||\log\varepsilon|+C.
  \end{align*}
Hence, we have $|a_{11}^{21}+a_{21}^{21}|\leq C|C_1^2-C_2^2||\log\varepsilon|+C.$
Similarly,
	\begin{align*}
	a_{11}^{23}+a_{21}^{23}=\int_{\Omega}\left(2\mu e({\bf u}_{1}^{2}+{\bf u}_2^2), e({\bf u}_{1}^{1})\right)+\int_{\Omega}{\bf u}\cdot \nabla{\bf u}\cdot {\bf u}_{1}^{3}\leq C|C_1^2-C_2^2||\log\varepsilon|+C.
\end{align*}
 For \eqref{esta11112}, by making use of \eqref{propuhe11}, Proposition \ref{propuhe11} and mean-value theorem, we get
\begin{align*}
	|a_{11}^{22}+a_{21}^{22}|&=\Big|\int_{\partial D_1}\boldsymbol{\psi}_2\cdot\sigma[{\bf u}_1^2+{\bf u}_2^2,p_1^2+p_2^2-(p_1^2+p_2^2)(R,x_2)]\Big|\\
	&\leq\,C|C_1^2-C_2^2| \int_{\partial D_1\cap\bar{\Omega}_R}\frac{|x_1|}{\delta(x_{1})^2}+\frac{C\sum_{\alpha=1}^3|C_1^\alpha-C_2^\alpha|}{\delta(x')}+C\\
	&\leq \frac{C|C_1^2-C_2^2|}{\varepsilon}+\frac{C\sum_{\alpha=1}^3|C_1^\alpha-C_2^\alpha|}{\sqrt{\varepsilon}}+C.
\end{align*}
We thus complete the proof.
\end{proof}

By the Cramer's rule we can solve $|C_1^\alpha-C_2^\alpha|$, and have
 
 \begin{prop}\label{lemCialpha}
	Let $C_{i}^{\alpha}$ be defined in \eqref{ud}. Then
	\begin{equation*}
		|C_1^1-C_2^1|\leq C\sqrt{\varepsilon},\quad |C_1^2-C_2^2|\leq C\varepsilon^{3/2},\quad|C_1^3-C_2^3|\leq C\sqrt{\varepsilon}.
	\end{equation*}
\end{prop}

\begin{proof}[Proof of Theorem \ref{mainthm2D}]
By using  Propositions \ref{lemma1}--\ref{lemCialpha}, Theorem \ref{mainthm2D} is proved.
\end{proof}

As an immediate consequence of Theorem \ref{mainthm2D}, we have the following estimate for the Cauchy stress tensor in dimension two.
\begin{corollary}\label{mainthmsigma2d}(Estimates of the Cauchy Stress)
Under the assumptions in Theorem \ref{mainthm2D}, we have  
	\begin{equation*}
		\inf_{c\in\mathbb{R}}|\sigma[{\bf u},p+c]|\leq 
		\frac{C}{\sqrt{\varepsilon}}\|{\boldsymbol\varphi}\|_{C^{2,\alpha}(\partial D)},\quad\quad~\text{in}~~\Omega_R.
	\end{equation*}
\end{corollary}

By using \cite[Theorem 1.6]{LX} and the same argument as in Theorem \ref{theolow}, we can derive the following lower bounds of $|\nabla {\bf u}(x)|$ in dimension two. 

\begin{theorem}(Lower Bounds in 2D)\label{theolow2d}
	Assume that $D=B_4(0)$, $D_1$ and $D_2$ be two unit balls $B_{1}(0,1+\frac{\varepsilon}{2})$ and $B_{1}(0,-1-\frac{\varepsilon}{2})$. Let ${\bf u}\in W^{1,2}(D;\mathbb R^{2})\cap C^2(\bar{\Omega};\mathbb R^{2})$ and $p\in L^2(D)\cap C^1(\bar{\Omega})$ be the solution to \eqref{sto}--\eqref{compatibility} with $\boldsymbol{\varphi}=(0,x_{2})$. Then for sufficiently small $0<\varepsilon<1/2$, 
	\begin{equation*}
		|\nabla {\bf u}(0,x_2)|\ge \frac{1}{C\sqrt{\varepsilon}},\quad\quad~~\text{for}~|x_{2}|\leq \varepsilon.
	\end{equation*}
\end{theorem}

\noindent{\bf Acknowledgements.} The work of H. Li was partially Supported by Beijing Natural Science Foundation (No.1242006), the Fundamental Research Funds for the Central Universities (No.2233200015), and National Natural Science Foundation of China (No.12471191).

\noindent{\bf Conflict of Interest.} The authors declare that they have no conflict of relevant financial or non-financial interest.

\noindent{\bf Availability of data and material} No data applicable.

\noindent{\bf Code availability} No code applicable.

\end{document}